\documentclass[10pt]{amsart}

\usepackage[utf8]{inputenc}
\usepackage[T1]{fontenc}	
\usepackage[french,english]{babel}

\usepackage[paperwidth=7in, paperheight=10in, margin=.875in]{geometry}
\usepackage[backref,colorlinks,linkcolor=red,anchorcolor=green,citecolor=blue]{hyperref}
\usepackage{amsfonts,amssymb}
\usepackage{amsmath}
\usepackage{graphicx} 
\usepackage{cite}
\usepackage{enumerate}

\usepackage{lmodern}			

\usepackage{placeins}
\usepackage{float} 

\usepackage{color}
\definecolor{vertFonce}{rgb}{0,0.5,0}
\definecolor{numLignes}{rgb}{0.17,0.57,0.7}	
\definecolor{gris}{rgb}{0.5,0.5,0.5}
\definecolor{grisFonce}{rgb}{0.2,0.2,0.2}
\definecolor{orange}{rgb}{1,0.65,0.31}		
\definecolor{orangeFonce}{rgb}{1,0.4,0}
\definecolor{bleuFonce}{rgb}{0,0,0.4}
\definecolor{rougeFonce}{rgb}{0.3,0,0}
\definecolor{rougeWord}{rgb}{0.5,0,0}
\definecolor{vertClair}{rgb}{0.8,1,0.8}
\definecolor{rougeClair}{rgb}{1,0.5,0.5}

\usepackage{pict2e}
\usepackage{dsfont}				
\usepackage{mathrsfs}
\usepackage{yfonts}

\newtheorem{lem}{Lemma}[section]
\newtheorem{theorem}{Theorem}
\newtheorem{cor}{Corollary}[section]
\newtheorem{prop}{Proposition}[section]
\newtheorem{definition}{Definition}[section]

\newtheorem{remark}{Remark}[section]

\newcommand{\step}[1]	{\paragraph{\itshape\bfseries Step #1.}}

\newcommand		{\ssi}		{\Leftrightarrow}

\newcommand		{\N}		{\mathbb N}			
\newcommand		{\R}		{\mathbb R}			
\newcommand		{\KK}		{\mathcal K}		
\newcommand     {\M}		{\mathcal{M}}		
\newcommand		{\PP}		{\mathcal{P}}		
\renewcommand	{\L}		{\mathcal L}		

\newcommand		{\ball}			{B}

\renewcommand		{\d}		{\mathrm{d}}	
\newcommand			{\dd}		{\,\d}			
\newcommand			{\dt}		{\frac{\d}{\d t}}	
\DeclareMathOperator{\F}		{\mathcal F}	
\DeclareMathOperator{\divg}		{div}			
\DeclareMathOperator{\G}		{\Gamma}		
\DeclareMathOperator{\E}		{\mathbb{E}}	

\newcommand		{\indic}		{\mathds{1}}	

\newcommand		{\lt}			{\left}			%
\newcommand		{\rt}			{\right}		%
\newcommand		{\lal}			{\langle}		%
\newcommand		{\ral}			{\rangle}		%
\newcommand		{\weight}[1]	{\lt\lal #1\rt\ral}	

\newcommand		{\intd}			{\int_{\R^d}}
\newcommand		{\iintd}		{\iint_{\R^{2d}}}

\newcommand		{\init}			{\mathrm{in}}
\newcommand		{\loc}			{\mathrm{loc}}
\newcommand		{\eps}			{\varepsilon}
\newcommand		{\e}			{\varepsilon}	
\newcommand		{\Eps}			{\mathcal{E}}
\newcommand		{\C}			{\mathcal{C}}
\newcommand		{\CS}			{\C^\mathrm{S}}
\newcommand		{\CHLS}			{\C^\mathrm{HLS}}
\newcommand		{\CGNS}			{\C^\mathrm{GNS}}

\newcommand		{\ak}			{\beta}
\newcommand		{\al}			{\alpha}
\DeclareMathOperator{\I}		{\mathit{I}}
\newcommand		{\Dp}[1]		{{\mathfrak{D}}_{#1}}

\usepackage{stmaryrd}
\newcommand		{\subsetArrow}	{\mathrel{\ooalign{$\subset$\cr%
\hidewidth\raise-.087ex\hbox{$_\shortrightarrow\mkern-1.5mu$}\cr}}}
\newcommand		{\subsetarrow}	{\mathrel{\ooalign{$\subset$\cr%
\hidewidth\raise-1.45ex\hbox{$\vec{}\mkern6mu$}\cr}}}

\title[Fractional Keller-Segel Equation]{Fractional Keller-Segel Equation:  Global Well-posedness and Finite Time Blow-up}
\author{Laurent Lafleche$^{1,2}$ \and Samir Salem$^1$}
\thanks{$^1$CEREMADE, UMR CNRS 7534, Université Paris-Dauphine, PSL Research University, Place du Maréchal de Lattre de Tassigny, 75775 Paris cedex 16 France, {\tt lafleche@ceremade.dauphine.fr} \and {\tt salem@ceremade.dauphine.fr}}
\thanks{$^2$CMLS, \'Ecole polytechnique, CNRS, Universit\'e Paris-Saclay, 91128 Palaiseau cedex, France}
\date{\today}

\begin{document}

\begin{abstract}
	This article studies the aggregation diffusion equation
	\[
		\partial_t\rho = \Delta^\frac{\alpha}{2} \rho + \lambda\,\mathrm{div}((K*\rho)\rho),
	\]
	where $\Delta^\frac{\alpha}{2}$ denotes the fractional Laplacian and $K = \frac{x}{|x|^\beta}$ is an attractive kernel. This equation is a generalization of the classical Keller-Segel equation, which arises in the modelling of the motion of cells. In the \textit{diffusion dominated} case $\beta < \alpha$, we prove global well-posedness for an $L^1_k$ initial condition, and in the \textit{fair competition} case $\beta = \alpha$ for an $L^1_k\cap L\ln L$ initial condition with small mass. In the \textit{aggregation dominated} case $\beta > \alpha$, we prove global or local well-posedness for an $L^p$ initial condition, depending on some smallness condition on the $L^p$ norm of the initial condition. We also prove that finite time blow-up of even solutions occurs under some initial mass concentration criteria.
\end{abstract}
 
 \maketitle
 \bigskip

\section{Introduction}\label{sec:intro}

	The models arising in the context of the chemotaxis of cells have been thoroughly studied in recent years. Among those, the (parabolic-elliptic) Keller-Segel equation models the competition between the aggregation and diffusion of cells (see \cite{blanchet_two-dimensional_2006} and references therein for a proper biological and mathematical introduction on the topic). In this paper we consider a variant of this classical model where the diffusion is modelled with a fractional Laplacian. Such a choice is biologically motivated (see for instance \cite{escudero_fractional_2006, bournaveas_one-dimensional_2010} and references therein). From a mathematical point of view, it is then interesting to study how such a diffusion competes with an aggregation field which singularity is up to the Newtonian one.

	More precisely for some $(\al,\ak)\in \R_+^2$, we consider the fractional Keller-Segel equation
	\begin{equation}\label{eq:FKS}\tag{FKS}
		\partial_t\rho  = \Delta^\frac{\al}{2} \rho + \lambda\divg((K*\rho)\rho),
	\end{equation}
	where $\lambda>0$ is a parameter encoding the chemosensitivity, or the intensity of the aggregation. The interaction kernel is given by
	\begin{equation*}
		K(x) := \frac{x}{|x|^\ak},
	\end{equation*}
	and $\I := \Delta^\frac{\al}{2}$ denotes the fractional Laplacian defined by
	\begin{equation}\label{def:lapfrac}
		\I(u) = \Delta^\frac{\al}{2} u := c_{d,\al}\intd \frac{u(y)-u(x)}{|x-y|^{d+\al}} \dd y.
	\end{equation}
	The constant $c_{d,\al}$ can be written $c_{d,\al} = -(2\pi)^\al \frac{\omega_{-\al}}{\omega_{d+\al}} > 0$ where $\omega_d = \frac{2\pi^{d/2}}{\G(d/2)}$ is the size of the unit sphere in $\R^d$ when $d\in\N^*$.

	Particular cases of equation~\eqref{eq:FKS} have been studied by numerous authors recently. The classical case corresponds to the choice $\al=\ak=d=2$ and has been thoroughly studied in the past years. In \cite{blanchet_two-dimensional_2006}, the authors show the global well-posedness when the initial mass $M_0$ is smaller than the critical one $M_c=\frac{4}{\lambda}$. Above this mass, a finite time blow-up is shown to appear. This blowup phenomenon was already proved in \cite{jager_explosions_1992} (see also \cite{nagai_blow-up_1995}). In \cite{carrillo_uniqueness_2014} is also established the well posedness for an $L^\infty$ initial condition. This assumption is sufficient to enjoy the Log-Lipschitz regularity of the nonlinear drift $K*\rho$, as in this case $K$ is the Newtonian kernel (see for instance \cite{loeper_uniqueness_2006}). It is possible to relax this assumption to $L\,\ln L$ initial data \cite{egana_fernandez_uniqueness_2016} or even measure initial data \cite{bedrossian_existence_2014}. Large time behaviour is also studied in \cite{blanchet_two-dimensional_2006, campos_serrano_asymptotic_2014, egana_fernandez_uniqueness_2016}. In higher dimension, the variant case $\al=2$, $\ak=d=3$ is studied in \cite{corrias_global_2004}, where a finite time blow-up is obtained under a concentration of initial mass condition.
	
	\begin{figure}[h]\centering
		\mbox{{\includegraphics[scale=0.75]{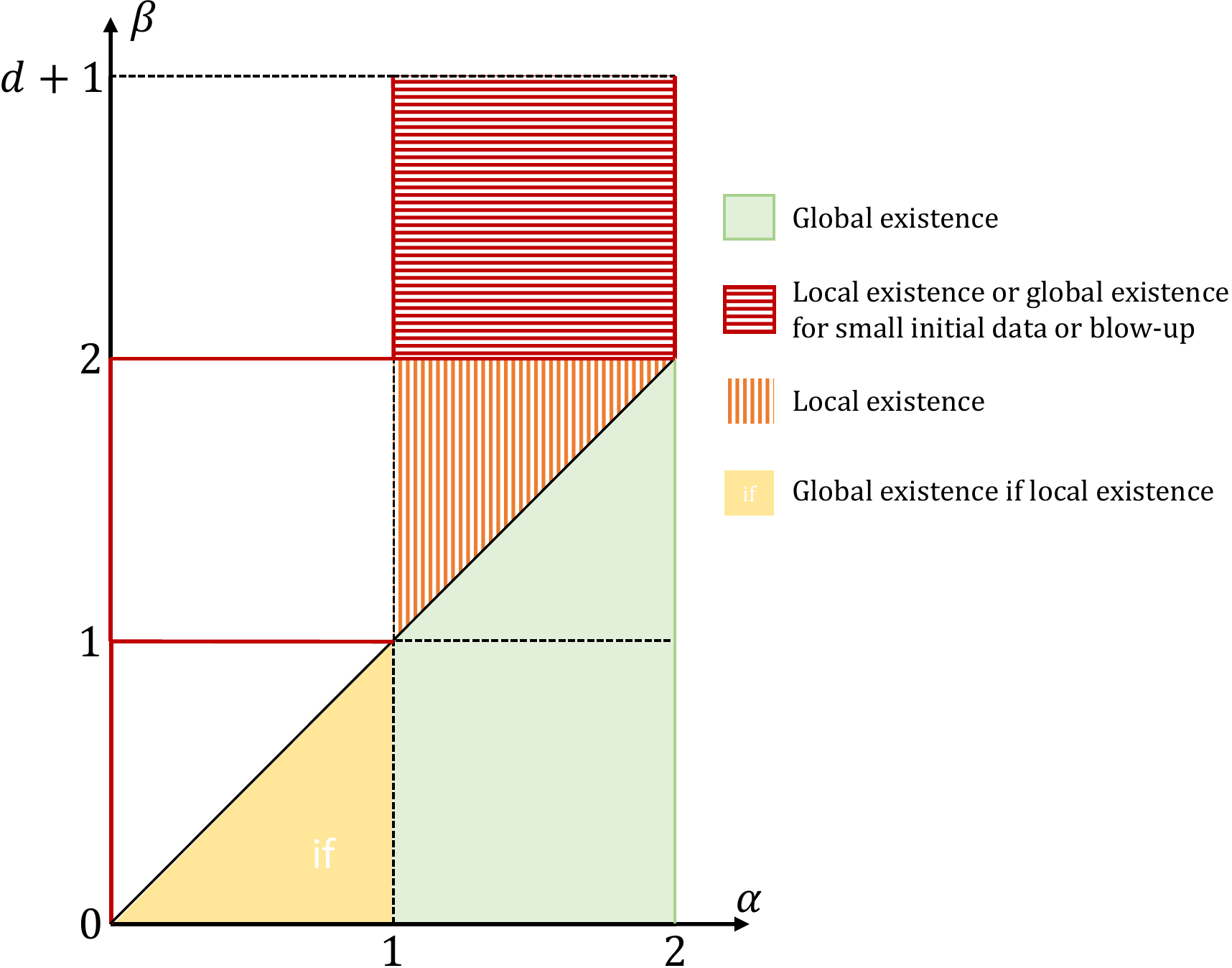}}}
		\caption{Existing results for the \eqref{eq:FKS}~equation.}
		\label{fig1}
	\end{figure}
	
	The literature on the fractional case $\al<2$, is also large and growing and previously known results are summarized in Figure~\ref{fig1}. In a significant part of it, the kernel $K$ is the Newtonian one ($\ak=d$). In the one dimensional case, \cite{bournaveas_one-dimensional_2010} provides a well posedness result for an $L^p$ initial condition with $p>\frac{1}{\al}$ when $\al\in (0,1)$ and $p>1$ when $\ak\in (0,1)$, as well as a finite time blow-up of even solutions under some concentration of initial mass criteria. The critical case $\al = 1$ was then treated in \cite{burczak_critical_2016}. In the case $d\geq 2$, \cite{biler_morrey_2016} also provides some concentration of initial mass criteria leading to a blow-up of solutions when $\al \in (0,2)$. See also the recent paper \cite{biler_large_2018} for sharper results. Still in the Newtonian case, \cite{li_exploding_2010} provides similar results in the range $\al \in (0,2)$. In the limiting case $\al=0$, see \cite{bertozzi_blow-up_2009} for $\ak\in [0,1)$, \cite{bertozzi_finite-time_2007} for $\ak=1$, and \cite{li_global_2010} for $\ak\in(1,2)$. For $\al=2$ and $\ak\in (0,2)$, see \cite{karch_blow-up_2011} and \cite{godinho_propagation_2015}, and for $\ak=1$ and $\al\in(0,1)$, see \cite{li_finite-time_2009, li_refined_2009, biler_blowup_2009}. For a wider class of parameters, see also \cite{salem_propagation_2017} of the second author and \cite{biler_global_1999, biler_blowup_2009}.

\section{Main Results}\label{sec:results}

	We summarize our results in the following Figure~\ref{fig2}.
	\begin{figure}[h]\centering
		\mbox{{\includegraphics[scale=0.75]{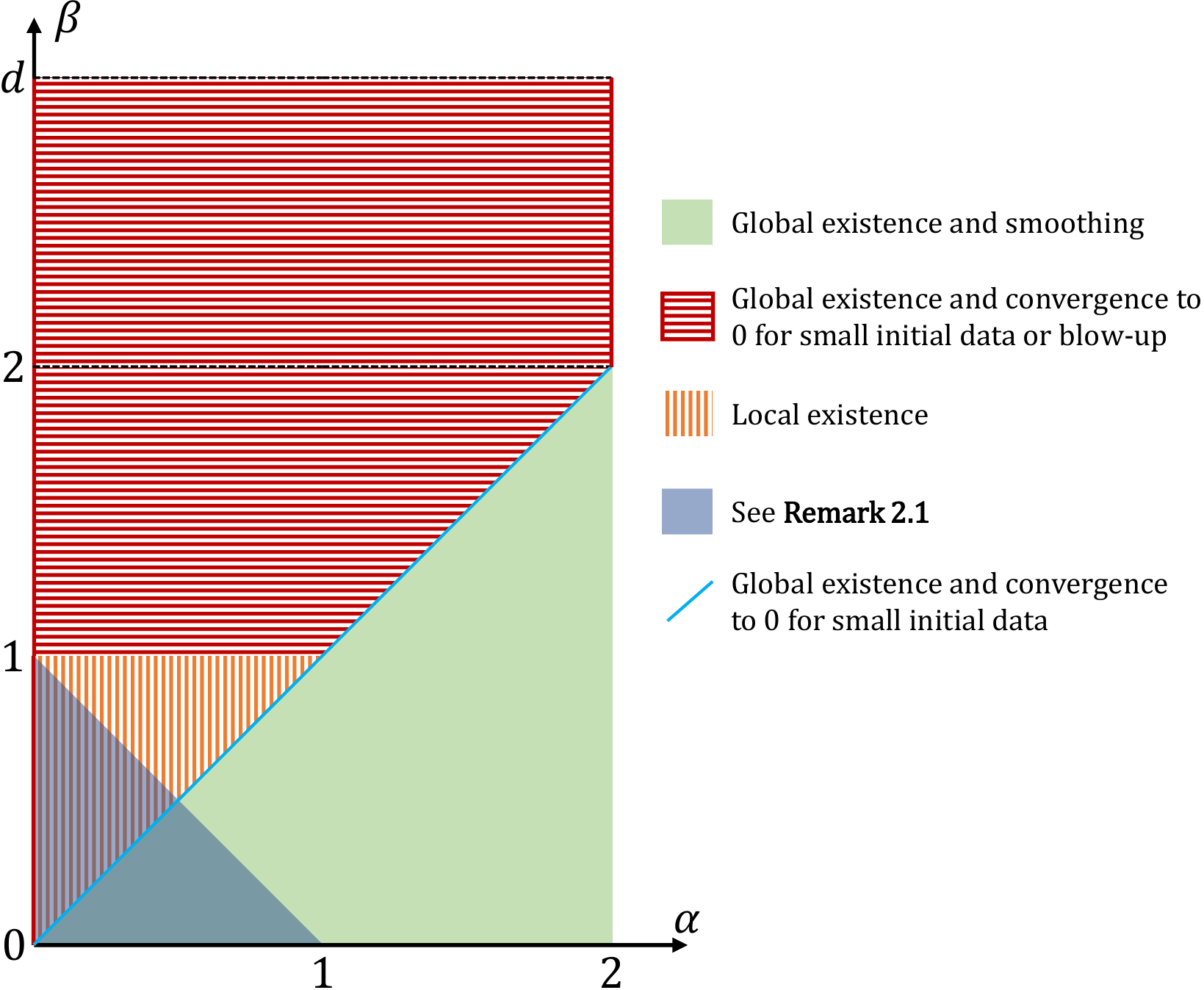}}}
		\caption{Range of application of Theorems~\ref{thm:welpos} and \ref{thm:BU}.  We emphasize that for $d>2$ the results extend to the segment $(\al,\ak)\in \{2\}\times(0,d)$.}
		\label{fig2}
	\end{figure}

	We will work on weighted spaces defined by
	\begin{align*}
		\M_k &:= \{\rho\in\M,\weight{x}^k\rho\in\M\}
		\\
		L^p_k &:= \{\rho\in L^p,\weight{x}^k\rho\in L^p\},
	\end{align*}
	where $\weight{x} = \sqrt{1+|x|^2}$, $L^p = L^p(\R^d)$ and $\M = \M(\R^d)$ denote the space of bounded measures. We also define the space of functions with finite entropy by
	\begin{eqnarray}
		L\,\ln L := \{\rho\in L^1, \rho\ln(\rho) \in L^1\}.
	\end{eqnarray}
	For $s\in (0,1)$, we will denote by $\CS_{d,s}$ the best Sobolev's constant such that for any $f\in H^s$ 
	\begin{equation*}
		\CS_{d,s}\|f\|^2_{L^{\frac{2d}{d-2s}}}\leq |f|^2_{H^s},
	\end{equation*}
	and for $\ak\in (0,d)$ and $p,q>1$ satisfying $2-\frac{\ak}{d}=\frac{1}{p}+\frac{1}{q}$, we will denote by $\CHLS_{d,\ak,p}$ the best Hardy-Littlewood-Sobolev's constant such that for any $f\in L^p$, $g\in L^q$,
	\begin{equation}\label{eq:HLS}
		\lt|\iintd |x-y|^{-\ak}f(x)g(y)\dd x\dd y\rt| \leq \CHLS_{d,\ak,p}\|f\|_{L^p}\|g\|_{L^q}.
	\end{equation}
	Finally for $s\in[0,d)$ and $r = \frac{2d}{d-s}$, we denote $\CGNS_{d,s}$ the best Gagliardo-Nirenberg-Sobolev's constant such that it holds
	\begin{equation*}
		\CGNS_{d,s} \|f\|_{L^r}^2 \leq \|f\|_{L^2} |f|_{H^s}.
	\end{equation*}
	
	For a given given couple $(\ak,\al)$ we define the following exponents for the $L^p$ spaces which will characterize the integrability of the density
	\begin{align}\label{def:p_ak_al}
		p_{\ak,\al} &:= \frac{d}{d+\al-\ak}
		\\\label{def:p_ak}
		p_\ak &:= p_{\ak,0} = \frac{d}{d-\ak}.
	\end{align}
	
	Taking $K = \frac{x}{|x|^\ak}$ let appear two main difficulties. The first one is the singularity at $x=0$ and the second is the behaviour when $x\to\infty$. We will therefore write
	\begin{equation*}
		K = K_0 + K_c = \chi K + (1-\chi)K,
	\end{equation*}
	where $\chi \in C^\infty_c$ verifies $\mathds{1}_{B_1} \leq \chi \leq \mathds{1}_{B_2}$. Several parts of our analysis could be easily generalized to more general kernels with similar behaviour. 

	\begin{definition}\label{def:sol}
		For any $T>0$, we say that $\rho$ is a weak solution to the \eqref{eq:FKS} equation on $(0,T)$ with initial condition $\rho^\init \in \mathcal{M}$ if it satisfies
		\begin{align*}
			\rho &\in C^0\lt([0,T), \M_{(1-\ak)_+}\rt) &&\text{if } \ak\in(0,2]
			\\
			\rho &\in C^0\lt([0,T), \M\rt)\cap L^1_\loc\lt((0,T),L^{p_{\ak,2}}\rt) &&\text{if } \ak\in(2,d+2),
		\end{align*}
		and for any $\varphi\in C^2_c$ 
		\begin{align}\label{eq:varFKS}
			\intd \lt(\rho(t)-\rho^\init\rt)\varphi = \int_0^t &\intd \rho(s)  \lt(I(\varphi) - K_c*(\rho(s)\cdot\nabla\varphi)\rt)
			\\\nonumber
			&+ \iintd K_0(x-y)(\nabla\varphi(x)-\nabla\varphi(y))\rho(s,\d x)\rho(s,\d y)\dd s.
		\end{align}
		We say that this solution is global if we can take $T=+\infty$.
	\end{definition}
	The definition makes sense since it is easy to notice that 
	\begin{align*}
		K_c*(\rho\nabla\varphi) &\in C^0 \cap L^\infty(\weight{x}^{\ak-1})
		\\
		K_0(x-y)(\nabla\varphi(x)-\nabla\varphi(y)) &\in C^0 \cap L^\infty(\R^{2d}) &&\text{if }\ak\in(0,2).
	\end{align*}
	Moreover, if $\ak\in(2,d+2)$, the last term in Definition~\ref{def:sol} is bounded thanks to Hardy-Littlewood-Sobolev inequality. Remark that at least formally, this equation conserves the total mass which we will denote by
	\begin{equation*}
		M_0 := \intd \rho^\init.
	\end{equation*}
	
	First we obtain a global or local well-posedness result, depending on the regime, given in the
	\begin{theorem}\label{thm:welpos}
		Let $(\al,\ak)\in [0,2)\times [0,d)$ be such that $\ak + \al > 1$ and $k\in [(1-\ak)_+,\al)$.

		\noindent$\bullet$ When $\ak<\al$ and $\rho^\init\in L^1_k$, there exists a unique and global weak solution to the \eqref{eq:FKS} equation. 
		
		\noindent$\bullet$ When $\ak=\al$, if $\rho^\init\in L^1_k\cap L\ln L$ satisfies 
		\begin{equation}\label{eq:condSM}
			\lambda M_0 < C_{\ak,d} = \frac{4(2\pi)^{\ak}}{(d-\ak)}\lt(\frac{\omega_{2d}}{\omega_d}\rt) \frac{\omega_{d-\ak}}{\omega_{2d-\ak}} \max\lt(\frac{\omega_{d-\ak}}{\omega_{d+\ak}},\frac{\omega_{d-\ak/2}^2}{\omega_{d+\ak/2}^2}\rt),
		\end{equation}
		then there exists a unique and global weak solution to the \eqref{eq:FKS} equation.
		
		\noindent$\bullet$ When $\ak>\al$ and $\rho^\init\in L^1_k\cap L^p$ with $p\in (p_{\ak,\al}, p_\ak)$, there exists a time $T>0$ such that there is a unique solution to the \eqref{eq:FKS} equation on $(0,T)$. Moreover there is a constant $C_{\lambda,p}(M_0)$ such that if 
		\begin{equation}\label{eq:condEta}
			\|\rho^\init\|_{L^p}\leq C_{\lambda,p}(M_0),
		\end{equation}
		then the solution is global.
	\end{theorem}
	
	\begin{remark}
		\label{rq:TriBer}
		The constraint $\ak + \al > 1$ comes from the necessity to propagate moments, which is necessary for our notion of solution and gives us compactness. Remark that it is only due to the behaviour at infinity of the interaction kernel, which we denoted by $K_c$, and not to the singularity. Therefore, our theorem would hold also for example for the following kernel
		\begin{equation*}
			K(x) = \frac{x}{|x|^\ak}\chi(x) + \frac{x}{|x|^\gamma}(1-\chi(x)),
		\end{equation*}
		for any $\gamma > 1 - \al$ and which relaxes the condition $\ak + \al > 1$. It is interesting also to notice that formula~\eqref{eq:formulation_sans_moments} could also provide an alternative definition of solution which does not need moments. However, it is not clear whether it is sufficient to provide compactness.
	\end{remark}
	
	\begin{remark}
		The explicit value of $\CHLS_{d,\ak,p}$ for $\ak\in (0,d)$ and $p=q$ in \eqref{eq:HLS} and $\CS_{d,s}$ for $s\in(0,1)$ are known, see for instance \cite{lieb_sharp_1983, lieb_analysis_2001}. Remarking that the HLS conjugate as defined in \eqref{eq:HLS} of $p_{\ak/2}$ is itself, it holds
		\begin{align*}
			\CHLS_{d,\ak,p_{\ak/2}} &=\pi^\frac{\ak}{2} \frac{\G\lt(\frac{d-\ak}{2}\rt)}{\G\lt(d-\frac{\ak}{2}\rt)} \lt(\frac{\G\lt(\frac{d}{2}\rt)}{\G(d)}\rt)^{-1+\frac{\ak}{d}} = \frac{\omega_{2d-\ak}}{\omega_{d-\ak}} \lt(\frac{\omega_{2d}}{\omega_{d}}\rt)^\frac{\ak-d}{d}
			\\
			\CS_{d,s} &= \frac{2^{2s}\pi^s\G\lt(\frac{d+2s}{2}\rt)}{\G\lt(\frac{d-2s}{2}\rt)}\lt(\frac{\G\lt(\frac{d}{2}\rt)}{\G(d)}\rt)^{\frac{2s}{d}} \ \, = (2\pi)^{2s} \frac{\omega_{d-2s}}{\omega_{d+2s}} \lt(\frac{\omega_{2d}}{\omega_{d}}\rt)^\frac{2s}{d},
		\end{align*}
		where we recall that $\omega_d = \frac{2\pi^{d/2}}{\G(d/2)}$.
	\end{remark}
	
	In the case $\ak \leq \al$, this theorem enlarges the existing result by Biler et al. \cite{biler_global_1999}, where global existence is proved for $d=2,3$ in the case $\al\leq \frac{d}{2}$, and is a novelty in higher dimension. Also it is provided with larger class of initial condition, and a uniqueness result. Note that the case $\al=\ak$ is only the object of some remark in \cite[Remark 3.2]{biler_global_1999}. As for the case $\al < \ak < 2$, it seems it has not been treated yet to the best of the authors' knowledges. See also \cite{biler_morrey_2016} and \cite{li_exploding_2010} for the case $\ak = 2$.
	
	Let us briefly sketch the proof of this theorem in the case of an $L\ln L$ initial condition. Formally differentiating the Boltzmann's entropy along \eqref{eq:FKS} (see for instance \cite[Section 2.2]{blanchet_two-dimensional_2006}) provides a control of the $L^1([0,T),L^p)$ for $p\in [1,p_\alpha]$ by fractional Sobolev's embedding, for any initial mass in the \textit{diffusion dominated} case and for small initial mass in the \textit{fair competition} case. Then a slight modification of standard coupling argument enables to obtain stability in this space when $p\in [1,p_a)$ and uniqueness when $p=p_a$. The other assumption on the initial condition are meant to control the $L^1([0,T),L^{p_\ak})$ norm of the solution in the different regimes. 

	When global existence holds, we also retrieve some additional properties as a quantitative rate of convergence to $0$ in the \textit{aggregation dominated} case and a gain of local integrability in the \textit{diffusion dominated} case.
	\begin{theorem}\label{thm:behaviour}
		Let $(\al,\ak)\in[0,2)\times [0,d)$ and $\rho$ be a solution of the \eqref{eq:FKS} equation as given by Theorem~\ref{thm:welpos}.
		
		\noindent$\bullet$ When $\ak < \al$, the gain of integrability is given for any $p\in(1,p_\ak)$ by
		\begin{equation*}
			\|\rho\|_{L^p} \leq C M_0\, t^{-\frac{d}{\al q}} + C_\lambda(M_0).
		\end{equation*}
		
		\noindent$\bullet$ When $\al < \ak$ and for a given $p\in(p_{\ak,\al},p_\ak)$, $\|\rho^\init\|_{L^p} < C_{\lambda,p}(M_0)$ defined by \eqref{eq:condEta}, then there exists a constant $C = C_{\ak,\al,p}(\rho^\init) > 0$ such that 
		\begin{equation*}
			\|\rho\|_{L^p} \leq CM_0\, t^{-\frac{d}{\al q}}.
		\end{equation*}
		
		\noindent$\bullet$ When $\ak = \al$, the condition becomes
		\begin{equation*}
			\lambda M_0 < C_{\ak,d,p} = \frac{4\CS_{d,\ak/2}}{p(d-\ak)\CHLS_{d,\ak,r}}, \quad \frac{1}{r} := \frac{p}{p+1} \frac{1}{p} + \frac{1}{p+1}\frac{1}{p_a},
		\end{equation*}
		which gives both a gain of integrability and an asymptotic behaviour for any $p\in(1,p_\ak)$
		\begin{align}\label{eq:condEtaFC}
			\lambda M_0 &\leq C_{\ak,d,p} \implies \|\rho\|_{L^p} \leq C M_0 \,t^{-\frac{d}{\al q}},
		\end{align}
		where $C$ depends only on $M_0$, $d$, $p$, $\ak$ and $\al$.
	\end{theorem}
	
	\begin{remark}
		If $\rho$ is a weak solution to the \eqref{eq:FKS} equation as given by definition~\ref{def:sol} with $\ak=\al$ and $\lambda M_0<C_{d,\ak,p}$ for a given $p>1$, we are not able to assert the uniqueness unless we assume that $\rho^\init\in L\ln L$. 
	\end{remark}
	
	Finally we obtain a finite time blow-up for even solutions to \eqref{eq:FKS} under some concentration of mass condition stated in the
	\begin{theorem}\label{thm:BU}
		Let $(\al,\ak)\in [0,2)\times [1,d)$ be such that $\al < \ak$, $k\in(0,\al)$ and $\rho\in C^0(\R_+,L^1_k)$ be an even nonnegative weak solution to the \eqref{eq:FKS} equation with initial condition $\rho^\init \in L^1_k$ verifying
		\begin{align}\label{eq:condBU_1}
			\intd\rho^\init(x) \weight{x}^k\dd x &\leq C^* \lambda^\frac{k}{2(\ak-k)} M_0^\frac{2\ak-k}{2(\ak-k)} &&\text{if } \al > 1
			\\\label{eq:condBU_2}
			\intd \rho^\init(x) |x|^k\dd x &\leq C^*_2M_0 \text{   and   } \lambda M_0 \geq C^*_3 &&\text{if } \al < 1
		\end{align}
		for given constants $C^*$, $C^*_2$, $C^*_3$ depending only on $d$, $\ak$, $\al$ and $k$. Then the solution ceases to exist in finite time.
	\end{theorem}
	
	The proof of this theorem relies on the time differentiation of an adequate moment, which is adapted to the fractional diffusion and not Newtonian aggregation case, and which leads to a contradiction.
	
	One of the strength of the result of Theorem~\ref{thm:BU}, even if it deals only with even solutions, is that it applies to weakly singular interactions, i.e. $\ak < 2$. Indeed it seems that so far most of finite time blow-up results for aggregation fractional diffusion equation dealt with the case of a Newtonian interaction at the exception of \cite[Theorem~2.2]{biler_blowup_2009}, which deals with interactions of the from $\frac{x}{|x|}$ near the origin. Considering a less singular kernel than the Newtonian erases some algebraic facilities and requires a thinner estimation of the competing terms. We emphasize that it also covers the purely aggregative case $\al=0$, giving stronger results than \cite{bertozzi_finite-time_2007, li_global_2010} for the case $\ak\geq 2$. For $\ak\leq 2$, the blow-up was already proved in \cite{bertozzi_blow-up_2009} using a Lagrangian point of view.
	
	Finally, let us comment about the disjunction of the different global existence and finite time blow-up conditions. Condition~\eqref{eq:condEta} in Theorem~\ref{thm:welpos} is heuristically in contradiction with the assumption of Theorem~\ref{thm:BU}. First remark that if we require that $\rho^\init$ is concentrated around zero, for instance with a condition of the type $\|\rho^\init\|_{L^1_k} < C M_0$ for a given constant $C$ which does not depend on $\rho^\init$, then the condition of blow-up~\eqref{eq:condBU_1} is equivalent to
	\begin{equation*}
		\lambda M_0 \geq C',
	\end{equation*}
	where $C'$ is a positive constant that depends only on $\ak$, $\al$, $k$ and $d$. Moreover, in a more general setting, for $k>0$, $q=p'\in (1,\infty)$ and $\rho\in L^1_k\cap L^p$, the following inequality
	\begin{equation*}
		\intd \rho \leq C \lt( \intd \rho\weight{x}^k \rt)^\frac{d}{d+kq}\|\rho\|_{L^p}^\frac{kq}{d+kq},
	\end{equation*}
	holds with $C$ depending only on $d$, $k$ and $q$. With fixed $M_0$, this inequality is enough to exclude a priori \eqref{eq:condEta} from \eqref{eq:condBU_1} or \eqref{eq:condBU_2}, at least in the range of arbitrarily large (or small) $\|\rho^\init\|_{L^p}$ or $\intd \rho^{\init}\weight{x}^k$. When this is not the case, we expect that no other behaviour appear in the remaining cases.
	
	\begin{figure}[h]\centering\label{fig3}
	 	\mbox{{\includegraphics[scale=0.7]{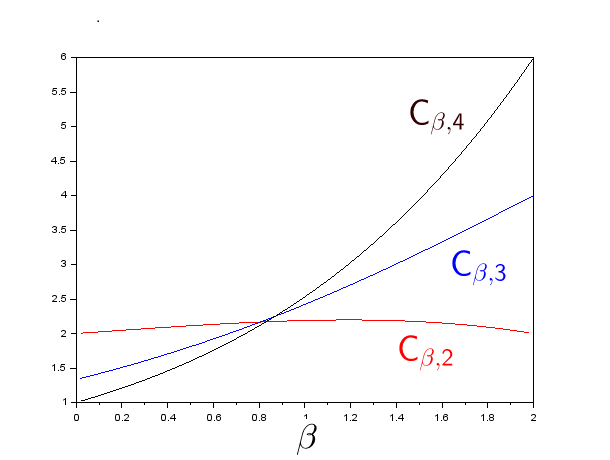}}}
	 	\caption{Lower bound of the threshold of condition~\eqref{eq:condSM} for $d=2,3,4$ and $\ak\in(0,2)$. For the case $\ak\leq \frac{1}{2}$ see Remark~\ref{rq:TriBer}.}
	\end{figure}
	
	We restrict ourselves to check that in the simple case $\al=\ak=2<d$, the global well-posedness condition \eqref{eq:condSM} is coherent with the classical large mass blow-up criteria. Indeed take a solution to \eqref{eq:FKS} in that case, it is possible to consider initial condition $\rho^\init\in L_2^1$ and then classically
	\begin{align*}
		\dt \intd \rho\, |x|^2 &=  \intd \rho \Delta (|x|^2) - \lambda \iintd K(x-y)\cdot(x-y)\rho(\d x)\rho(\d y)
		\\
		&= 2dM_0 - \lambda M_0^2
		\\
		&= 2dM_0 \lt(1 - \frac{\lambda M_0}{2d}\rt),
	\end{align*}
	so that the condition $\lambda M_0>2d$ yields to final time blow-up. And since $\omega_{a+2} = \frac{2\pi}{a}\omega_{a}$, it holds
	\begin{align*}
		C_{2,d}&=\frac{4(2\pi)^2}{(d-2)} \frac{\omega_{2d}}{\omega_d} \frac{\omega_{d-2}}{\omega_{2d-2}} \max\lt(\frac{\omega_{d-2}}{\omega_{d+2}},\frac{\omega_{d-1}^2}{\omega_{d+1}^2}\rt)
		\\
		&=\frac{4(2\pi)^2}{(d-2)} \frac{d-2}{2d-2} \max\lt(\frac{d(d-2)}{(2\pi)^2},\frac{(d-1)^2}{(2\pi)^2}\rt)
		\\
		&= 2(d-1) < 2d,
	\end{align*}
	so that the two conditions cannot be realized simultaneously.

\section{Proof of Theorem~\ref{thm:welpos} and Theorem~\ref{thm:behaviour}}\label{sec:welpos}

\subsection{A Priori estimates.}
	We begin this section with an a priori moment estimate given in the
	\begin{prop}[Propagation of weight]\label{prop:propag_m}
		Assume $1-\ak < \al$ and $\ak<2$ if $\al <1$ and let $k\in [(1-\ak)_+,\al)$ and $\rho$ be a solution of the \eqref{eq:FKS} equation with initial condition $\rho^\init\in L^1_k$. Then 
		\begin{equation*}
			\rho \in L^\infty_{\mathrm{loc}}(\R_+,L^1_k).
		\end{equation*}
	\end{prop}
	
	\begin{proof}
		Let $m=\weight{x}^k$ and $M_k = \|\rho\|_{L^1_k}$. When $k\geq 1$, the convexity of $m$ leads to
		\begin{align}\label{eq:d_t_moment}
			\frac{\d M_k}{\d t} &= \intd \rho \I(m) - \lambda\iintd h_m(x,y)\rho(\d x)\rho(\d y)
			\\\nonumber
			&\leq \intd\rho \I(m),
		\end{align}
		where $h_m(x,y) = \frac{(\nabla m(x)-\nabla m(y))\cdot(x-y)}{2|x-y|^\ak} \geq 0$. From \cite[Remark 4.2]{biler_blowup_2010} and \cite[Proposition 2.2]{lafleche_fractional_2018}, we know that for any $k\in(0,\al)$, 
		\begin{equation}\label{eq:I_m}
			\I(m) \leq C_{\al,k} m(x)\weight{x}^{-\al}.
		\end{equation}
		Since $m(x)\weight{x}^{-\al}\leq 1$, the following inequality holds
		\begin{align*}
			\frac{\d M_k}{\d t} &\leq C_{\al,k} M_0.
		\end{align*}
		When $k\in [1-\ak,\al\wedge 1)$, we decompose the second term in \eqref{eq:d_t_moment} as the sum of the integral over the domain $|x-y|<R$ and its complementary for a given $R>0$. Since $\nabla m$ is Lipschitz, we obtain
		\begin{align*}
			- \iint_{|x-y|\leq R} h_m(x,y)\rho(\d x)\rho(\d y)&\leq C\iint_{|x-y|\leq R} {|x-y|^{2-\ak}}\rho(\d x)\rho(\d y)
			\\
			&\leq CR^{2-\ak}M_0^2,
		\end{align*}
		where $C = \|\nabla^2 m\|_{L^\infty}$. The other part can be controlled as follows
		\begin{align*}
			- \iint_{|x-y|>R} h_m(x,y)\rho(\d x)\rho(\d y)&\leq k\iint_{|x-y|>R} \frac{(x\cdot y)(\weight{x}^{k-2}+\weight{y}^{k-2})}{|x-y|^\ak}\rho(\d x)\rho(\d y)
			\\
			&\leq 2k\iint_{|x-y|>R} \frac{(x\cdot y)\weight{x}^{k-2}}{|x-y|^\ak}\rho(\d x)\rho(\d y)
			\\
			&\leq 4k\iint_{|x-y|>R,|x|>|y|} \frac{|x|| y|\weight{x}^{k-2}}{|x-y|^\ak}\rho(\d x)\rho(\d y)
			\\
			&\leq 4k(I_1 + I_2),
		\end{align*}
		where
		\begin{align*}
			I_1 &= \iint_{|x-y|>R,2|x|<|y|} \frac{| y|\weight{x}^{k-1}}{|x-y|^\ak}\rho(\d x)\rho(\d y)
			\\
			I_2 &= \iint_{|x-y|>R,|x|<|y|<2|x|} \frac{| y|\weight{x}^{k-1}}{|x-y|^\ak}\rho(\d x)\rho(\d y).
		\end{align*}
		Since $|x-y|>||y|-|x||>|y||1-|x|/|y|| > |y|/2$ when $|y|>2|x|$, we get
		\begin{align*}
			I_1 &\leq 2^\ak\iint_{|x-y|>R,2|x|<|y|} | y|^{1-\ak}\weight{x}^{k-1}\rho(\d x)\rho(\d y)
			\\
			&\leq 2^\ak M_{1-\ak}M_0.
		\end{align*}
		For $I_2$, we use the fact that $|y|<2\weight{x}$ to obtain
		\begin{align*}
			I_2 &\leq \frac{2}{R^\ak}\iint_{|x-y|>R,|x|<|y|<2|x|}\weight{x}^{k}\rho(\d x)\rho(\d y)
			\\
			&\leq \frac{2}{R^\ak}M_k M_0.
		\end{align*}
		Combining these three inequalities with \eqref{eq:d_t_moment} and \eqref{eq:I_m}, we obtain
		\begin{align*}
			\frac{\dd M_k}{\dd t} &\leq C_{\al,k} M_{k-\al} + \lambda M_0 \lt(2^{1-\ak} M_{1-\ak} + \frac{2}{R^\ak} M_k + CR^{2-\ak}M_0\rt).
		\end{align*}
		In particular, since $1-\ak \leq k$ and $k-\al < 0$, we get
		\begin{align*}
			\frac{\dd M_k}{\dd t} &\leq  M_0 \lt(C_{\al,k,M_0} + \lambda\lt(2^{1-\ak} + \frac{2}{R^\ak}\rt) M_k\rt).
		\end{align*}
		By Gronwall's Lemma, this leads to
		\begin{align*}
			M_k &\leq \lt(M_k^\init + \frac{C_{\al,k,M_0}}{\lambda C_{\ak,R}}\rt) e^{\lambda C_{\ak,R} M_0t},
		\end{align*}
		which proves the result.
	\end{proof}
		
	The second type of estimates are a priori bounds of integrability. Let us first briefly emphasize that the quantities we estimate will take the form
	\begin{equation*}
		\intd \Phi(u(x))\dd x,
	\end{equation*}
	where $u\geq 0$ and $\Phi : \R_+ \to \R_+$ is a nondecreasing convex mapping such that $\Phi(0) = 0$ and $u\mapsto u\,\Phi''(u)\in L^1_\mathrm{loc}$. Then we can define 
	\begin{align}\label{eq:def_Psi}
		\Psi(u) &:= \int_0^u v\,\Phi''(v)\dd v
		\\\label{eq:def_psi}
		\psi(u) &:= \frac{1}{2}\int_0^u \sqrt{\Phi''}.
	\end{align}
	For $p=q'>1$ and $u\geq0$, we recover Lebesgue norms and Boltzmann's entropy as follow
	\begin{align*}
		\Phi_p(u) := \tfrac{1}{p-1}u^p \implies & \Psi_p(u) = u^p
		\\
		&\psi_p(u) = \tfrac{2}{\sqrt{p}} u^{p/2}
		\\
		\Phi_1(u) := u\ln(u) \implies & \Psi(u) = u
		\\
		&\psi_1(u) = 2 u^{1/2}.
	\end{align*}

	\begin{lem}[General estimate]
		\label{Genest}
		Assume that $(\al,\ak)\in (0,2]\times(0,d)$ (with $\al\neq 2$ if $d = 2$) and let $\rho$ be a smooth solution to the \eqref{eq:FKS} equation, $\Phi$ be a nondecreasing convex mapping, $\Psi$ and $\psi$ be defined respectively by \eqref{eq:def_Psi} and \eqref{eq:def_psi} and $b\in\lt(1,p_\ak\rt)$. Then there holds
		\begin{align}\label{eq:estim0}
			\dt \lt(\intd \Phi(\rho)\rt) &\leq \lambda \lt(d-\ak\rt)\CHLS_{d,\ak,b}\lt\|\rho\rt\|_{L^s}\|\Psi(\rho)\|_{L^b} - |\psi(\rho)|_{H^\frac{\al}{2}}^2,
			\\\label{eq:estim}
			&\leq \lambda \lt(d-\ak\rt)\CHLS_{d,\ak,b}\lt\|\rho\rt\|_{L^s}\|\Psi(\rho)\|_{L^b} - \CS_{d,\al/2} \|\psi(\rho)\|^2_{L^{\tilde{b}}},
		\end{align}
		where
		\begin{align*}
			\frac{1}{s} = 2-\frac{\ak}{d} - \frac{1}{b}, \quad	\frac{2}{\tilde{b}} = 1 - \frac{\al}{d}.
		\end{align*}
	\end{lem}
	
	\begin{proof}
		We define the "Carré du Champs" and the $\Phi$-dissipation by
		\begin{align}\label{def:G}
			\G(u,v) &:= \frac{c_{d,\al}}{2}\intd \frac{(u(y)-u(x))(v(y)-v(x))}{|x-y|^{d+\al}} \dd y\dd x
			\\
			\Dp{\Phi}(u) &:= \G(u,\Phi'(u)),
		\end{align}
		where $c_{d,\al}$ is defined in \eqref{def:lapfrac}. With these definitions, we have
		\begin{align*}
			\intd \I(u)v &= \intd u\I(v) = -\intd \G(u,v).
		\end{align*}
		In particular, since $\Phi$ is convex,
		\begin{equation*}
			\intd \I(u)\Phi'(u) = -\intd \Dp{\Phi}(u) \leq 0.
		\end{equation*}
		We remark that
		\begin{align*}
			|\psi(u)-\psi(v)|^2 &= \lt|\int_{u}^{v}\sqrt{\Phi''}\rt|^2
			\\
			&\leq \lt(\int_{u}^{v}\d s\rt)\lt(\int_{u}^{v}\Phi''\rt)
			\\
			&\leq (u-v)(\Phi'(u)-\Phi'(v)),
		\end{align*}
		which by definition \eqref{def:G} leads to
		\begin{align*}
			\G(\psi(u),\psi(u)) \leq \G(u,\Phi'(u)).
		\end{align*}
		Therefore
		\begin{equation}\label{eq:Strook-Varopoulos}
			|\psi(u)|^2_{H^\frac{\al}{2}} = \intd \G(\psi(u),\psi(u)) \leq \intd \G(u,\Phi'(u)) = \intd \Dp{\Phi}(u).
		\end{equation}
		Let $\rho$ be a nonnegative solution to the \eqref{eq:FKS} equation. Then formally
		\begin{align}\nonumber
			\dt \lt(\intd \Phi(\rho)\rt) &= \intd \Phi'(\rho) \I(\rho) - \lambda \Phi''(\rho)\nabla\rho\cdot(K*\rho) \rho
			\\\nonumber
			&= -\intd \Dp{\Phi}(\rho) - \intd \lambda \nabla(\Psi(\rho))\cdot (K*\rho)
			\\\nonumber
			&= -\intd \Dp{\Phi}(\rho) + \lambda \intd \Psi(\rho) (\divg(K) *\rho)
			\\\label{eq:formulation_sans_moments}
			&= -\intd \Dp{\Phi}(\rho) + \lambda\lt(d-\ak\rt) \intd  \lt(\frac{1}{|x|^{\ak}} *\rho\rt).
		\end{align}
		We remark that by Hardy-Littlewood-Sobolev inequality, we have
		\begin{align*}
			\lt(d-\ak\rt)\intd  \lt(\frac{1}{|x|^{\ak}} *\rho\rt) \Psi(\rho) &\leq \lt(d-\ak\rt)\CHLS_{d,\ak,b}\|\rho\|_{L^s}\|\Psi(\rho)\|_{L^b},
		\end{align*}
		and by \eqref{eq:Strook-Varopoulos} and Sobolev embeddings, we have
		\begin{equation*}
			-\intd \Dp{\Phi}(\rho) \leq -|\psi(\rho)|_{H^\frac{\al}{2}}^2 \leq - \CS_{d,\al/2} \|\psi(\rho)\|_{L^{\tilde{b}}}^2,
		\end{equation*}
		which ends the proof.
	\end{proof}
	
	\begin{prop}[$L\ln L$ estimate]\label{LlnLest}
		Let $\ak=\al$ and $\rho$ be a smooth function satisfying the \eqref{eq:FKS} equation with initial condition $\rho^\init \in L\ln L$. Then it holds
		\begin{equation*}
			\intd \rho\, \ln(\rho) +  4C_{\ak,d}^{-1} \lt(\lambda M_0 - C_{\ak,d}\rt) \int_0^t |\sqrt{\rho}|_{H^\frac{\ak}{2}}^2 \leq \intd \rho^\init \ln(\rho^\init),
		\end{equation*}
		with $\rho = \rho(t,\cdot)$ and
		\begin{equation*}
			C_{\ak,d} = \frac{4\lt(\CGNS_{d,\ak/2}\rt)^2}{\lt(d-\ak\rt)\CHLS_{d,\ak,p_{\ak/2}}}.
		\end{equation*}
		Moreover if $\lambda M_0 < C_{\ak,d}$ and for some $T,k>0$, $\rho\in L^\infty((0,T),L^1_k)$, then
		\begin{equation}\label{eq:integrabilite_en_temps}
			\rho \in L^1((0,T),L^{p_\ak}).
		\end{equation}
	\end{prop}
	
	\begin{remark}
		The explicit value for $\CGNS_{d,\ak/2}$ does not seem to be known, however the following lower bound holds
		\begin{equation}\label{eq:GNSlb}
			\CGNS_{d,s} \geq \max\lt(\CS_{d,s/2},(\CS_{d,s})^{1/2}\rt).
		\end{equation}
		Indeed, one way to get the Gagliardo-Nirenberg-Sobolev inequality is to first use Sobolev's inequality and then interpolation between $H^s$ spaces
		\begin{equation*}
			\CS_{d,s/2} \|f\|_{L^r}^2 \leq \lt|f\rt|_{H^\frac{s}{2}}^2 \leq \|f\|_{L^2} |f|_{H^s}.
		\end{equation*}
		Another way is to first interpolate between Lebesgue spaces and then to use Sobolev's inequality
		\begin{align*}
			(\CS_{d,s})^{1/2}\|f\|^2_{L^r} \leq (\CS_{d,s})^{1/2} \|f\|_{L^2} \|f\|_{L^{r_2}} \leq \|f\|_{L^2} \|f\|_{H^s},
		\end{align*}
		where $r_2 := \frac{2d}{d-2s}$.
	\end{remark}

	\begin{proof}
		We use inequality \eqref{eq:estim0} for $\Phi = \Phi_1$, $\psi_1(u) = 2u^{1/2}$ and $b = s = p_{\ak/2}$ to obtain
		\begin{equation*}
			\dt \intd \rho\,\ln(\rho) \leq  \lambda\lt(d-\ak\rt)\CHLS_{d,\ak,b}\lt\|\rho\rt\|^2_{L^b} - \lt|\psi_1(\rho)\rt|_{H^\frac{\ak}{2}}^2.
		\end{equation*}
		Then, by Gagliardo-Nirenberg-Sobolev's inequality, we have
		\begin{align*}
			\lt(\CGNS_{d,\ak/2}\rt)^2 \lt\|\rho\rt\|_{L^b}^2 &= \lt(\CGNS_{d,\ak/2}\rt)^2 \|\rho^{1/2}\|_{L^{2b}}^4
			\\
			&\leq \|\rho^{1/2}\|_{L^2}^2\,|\rho^{1/2}|_{H^\frac{\ak}{2}}^2 = M_0 |\rho^{1/2}|_{H^\frac{\ak}{2}}^2.
		\end{align*}
		Hence, since $\psi_1(u) = 2u^{1/2}$, we have
		\begin{equation*}
			4\lt(\CGNS_{d,\ak/2}\rt)^2 \|\rho\|^2_{L^b} \leq M_0 |\psi_1(\rho)|_{H^\frac{\ak}{2}}^2.
		\end{equation*}
		This yields
		\begin{equation*}
			\dt \intd \rho\, \ln(\rho) \leq C_{\ak,d}^{-1} \lt(\lambda M_0 - C_{\ak,d}\rt) |\psi_1(\rho)|_{H^\frac{\ak}{2}}^2,
		\end{equation*}
		which proves the first assertion. Formula~\eqref{eq:integrabilite_en_temps} comes form the fact for $k>0$, defining $m(x) := \weight{x}^k$ and $\lambda_k>0$ such that $\intd e^{-\lambda_k \weight{x}^k}\d x = 1$, with $h(u)=u\ln u -u+1\geq 0$ it holds
		\begin{align*}
			\intd \frac{\rho}{M_0}\, \ln \frac{\rho}{M_0} &= \intd h\lt(\frac{\rho}{M_0} e^{\lambda_k m}\rt) e^{-\lambda_km} + \intd \frac{\rho}{M_0} \ln (e^{-\lambda_km})\geq -\lambda_k \intd \frac{\rho}{M_0}m,
		\end{align*}
		and then 
		\begin{equation*}
			\intd \rho\ln \rho \geq M_0\,\ln M_0-\lambda_k \intd \rho\, m.
		\end{equation*}
		Combined with the following Sobolev inequality
		\begin{equation*}
			4\,\CS_{d,\ak/2} \|\rho\|_{L^{p_\ak}} = \CS_{d,\ak/2}\|\psi_1(\rho)\|_{L^{2p_\ak}}^2 \leq |\psi_1(\rho)|_{H^\frac{\ak}{2}}^2,
		\end{equation*}
		it yields
		\begin{multline*}
			0 \leq \intd (\rho \ln(\rho) + \lambda_k\lal x\ral^k\rho) - M_0\ln M_0 + 4\CS_{d,\ak/2} C_{\ak,d}^{-1} \lt(C_{\ak,d} -\lambda M_0 \rt) \int_0^t\|\rho\|_{L^{p_\ak}}
			\\
			\leq \intd \rho^\init \ln(\rho^\init) +\lambda_k\|\rho\|_{L^\infty(0,T;L^1_k)}-M_0\ln M_0,
		\end{multline*}
		and the conclusion follows.
	\end{proof}
	
	\begin{prop}[$L^p$ estimates]
		Let $(\al,\ak)\in[0,2)\times [0,d)$. Then, when $\ak < \al$ and $p = q'\in (1,p_\ak)$, we get a gain of integrability from $L^1$ to $L^p$ and a global in time propagation of the $L^p$ norm
		\begin{equation}\label{eq:gain_Lp}
			\|\rho(t)\|_{L^p} \leq C M_0 \max\lt(t^{-\frac{d}{\al q}}, M_0^\frac{d}{q(\al-\ak)}\rt),
		\end{equation}
		where $C >0$ is a constant depending on $d$, $\ak$, $\al$, $p$ and $\lambda$. When $\ak > \al$, then for any $p\in(p_{\ak,\al},p_\ak)$, there exists two constants $C = C_{\ak,\al,p} > 0$ and $C^\init = C_{\ak,\al,p}(\|\rho^\init\|_{L^p})$ such that
		\begin{align}\label{eq:etalement}
			\|\rho^\init\|_{L^p} &< C M_0 (\lambda M_0)^{-\frac{d}{(\ak-\al)q}} \implies \|\rho\|_{L^p} \leq C^\init M_0t^{-\frac{d}{\al q}}
			\\\label{eq:estim_Lp_temps_court}
			\|\rho^\init\|_{L^p} &> C M_0 (\lambda M_0)^{-\frac{d}{(\ak-\al)q}} \implies \rho \in L^\infty((0,T),L^p)
			\\
			\|\rho^\init\|_{L^p} &= C M_0 (\lambda M_0)^{-\frac{d}{(\ak-\al)q}} \implies \rho \in L^\infty(\R_+,L^p),
		\end{align}
		where $T < C_{\ak,\al,p}(\lambda,M_0) \|\rho^\init\|_{L^p}^{-pb}$ with
		\begin{align*}
			b &= \frac{\al}{p\lt(\al-\ak\rt) + d(p-1)}.
		\end{align*}
		When $\ak = \al$, then there exists a constant
		\begin{equation*}
			C_{\ak,d,p} = \frac{4\,\CS_{d,\al/2}}{ (d-\ak) \CHLS_{d,\ak,r}},
		\end{equation*}
		such that for any $p\in(1,p_\ak)$,
		\begin{align}\label{eq:etalement_critique}
			\lambda M_0 &\leq C_{\ak,d,p} \implies \|\rho\|_{L^p} \leq M_0(C^\init b)^{-\frac{1}{b}} t^{-\frac{d}{\al q}}
			\\\label{eq:estim_critique_temps_court}
			\lambda M_0 &\geq C_{\ak,d,p} \implies \rho \in L^\infty((0,T),L^p),
		\end{align}
		where $C^\init$ is a nonnegative constant depending on the initial data and
		\begin{equation*}
			T> \frac{1}{b\,C^\init} \lt(\frac{M_0}{\|\rho^\init\|_{L^p}}\rt)^\frac{\al q}{d}.
		\end{equation*}
	\end{prop}
	
	\begin{remark}
		The critical mass is clearly not optimal since we could use optimal constants in the Gagliardo-Nirenberg type embeddings, as it is done in the $L\ln L$ estimate, instead of using Sobolev's embeddings and interpolation between Lebesgue spaces.
	\end{remark}
	
	\begin{proof} We will separate the proof into several steps.
		\step{1. Differential inequality for the $L^p$ norm}
		We recall that
		\begin{equation*}
			\frac{1}{r} = \frac{p}{p+1} \frac{1}{p} + \frac{1}{p+1}\frac{1}{p_a}.
		\end{equation*}
		Since $p<p_a$, it implies that $r\in(p,p_a)$ and in particular $r/p >1$. Therefore, by taking $\Phi = \Phi_p$, $r=s$ and $b=r/p$ in inequality~\eqref{eq:estim} and defining $\tilde{r} =\frac{p\tilde{b}}{2}$, we obtain
		\begin{align}\label{eq:ineq_d_t}
			\dt \lt(\intd \Phi_p(\rho)\rt) &\leq \lambda\, \C_{\ak,r} \|\rho\|_{L^r}^{p+1} - \tfrac{\C_\al}{p} \|\rho\|_{L^{\tilde{r}}}^p,
		\end{align}
		where $\C_{\ak,r} = (d-\ak)\CHLS_{d,\ak,r}$, $\C_\al = 4\CS_{d,\al/2}$ and
		\begin{align}\label{eq:HLS_2}
			\frac{p+1}{r} &= 2-\frac{\ak}{d}
			\\\label{eq:Sobolev_2}
			\frac{p}{\tilde{r}} &= 1-\frac{\al}{d}.
		\end{align}
		We also remark that
		\begin{align*}
			r\leq\tilde{r} &\ssi \tfrac{1}{p}\lt(1-\tfrac{\al}{d} \rt) \leq \tfrac{1}{p+1}\lt(2-\tfrac{\ak}{d} \rt)
			\\
			&\ssi \lt(1+\tfrac{1}{p}\rt)\lt(d-\al\rt) \leq \lt(2d-\ak\rt)
			\\
			&\ssi p \geq \frac{d-\al}{d+\al-\ak}.
		\end{align*}
		Since $p \geq p_{\ak,\al} \geq \frac{d-\al}{d+\al-\ak}$, we deduce that $r\leq \tilde{r}$.
		
		We will now use interpolation between Lebesgue spaces to express the left hand side of \eqref{eq:ineq_d_t} in terms of $M_0$ and the $L^p$ norm only. Let $\eps \in (0,1)$ to be chosen later and
		\begin{align}\label{def:b_0}
			b_0 &:= \frac{\ak-\al(1-\eps)}{\eps d (p-1)} = \frac{\al}{d (p-1)} + \frac{\ak-\al}{\eps d (p-1)}
			\\\label{def:theta_1}
			\theta_1 &:= \frac{\eps p}{p+1}(1+b_0)
			\\\label{def:theta_2}
			\theta_2 &:= \frac{(1-\eps)p}{p+1}
			\\\label{def:theta_0}
			\theta_0 &:= 1 - \theta_1 - \theta_2.
		\end{align}
		Since $p>1$ and $\eps \in (0,1)$, we deduce that $\theta_2 \in [0,1)$.
		Moreover, using the respective definitions \eqref{eq:HLS_2} and \eqref{eq:Sobolev_2} of $r$ and $\tilde{r}$, we have
		\begin{align*}
			\frac{\theta_1}{p} + \frac{\theta_2}{\tilde{r}} + \theta_0 &= \frac{\eps (1+b_0)p}{p+1} \lt(\frac{1}{p}-1\rt) + \frac{(1-\eps)p}{p+1} \lt(\frac{1}{\tilde{r}}-1\rt) + 1
			\\
			&= \frac{1}{p+1} \lt(\eps (1-p)(1+b_0) + (1-\eps) \lt(1-\frac{\al}{d}-p\rt) + p+1\rt)
			\\
			&= \frac{1}{p+1} \lt(\eps (1-p) - \frac{\ak-\al(1-\eps)}{d} + 2 - \frac{\al}{d} -\eps \lt(1-p\rt) + \eps\frac{\al}{d}\rt)
			\\
			&= \frac{1}{p+1} \lt(2 - \frac{\ak}{d}\rt) = \frac{1}{r}.
		\end{align*}
		Therefore, if we can choose $\eps \in(0,1)$ such that $(\theta_0,\theta_1)\in[0,1]^2$, we obtain by interpolation
		\begin{equation*}
			\|\rho\|_{L^r}^{p+1} \leq M_0^{\theta_0(p+1)} \|\rho\|_{L^p}^{p\eps(1+b_0)} \|\rho\|_{L^{\tilde{r}}}^{p(1-\eps)} = A^{\eps}B^{1-\eps}.
		\end{equation*}
		Then, by using the standard Young inequality $a^{\eps}b^{1-\eps} \leq \eps a + (1-\eps) b$, for any $\eps_0>0$, we have
		\begin{equation*}
			A^{\eps}B^{1-\eps} = \lt(\lt(\frac{1-\eps}{\eps_0}\rt)^\frac{1-\eps}{\eps}A\rt)^{\eps} \lt(\frac{\eps_0B}{1-\eps}\rt)^{1-\eps} \leq C_{\eps,\eps_0} A + \eps_0 B,
		\end{equation*}
		with $C_{\eps,\eps_0} = \eps \lt(\frac{1-\eps}{\eps_0}\rt)^\frac{1-\eps}{\eps}$. Coming back to \eqref{eq:ineq_d_t}, it yields
		\begin{align}\label{eq:ineq_d_t_2}
			\dt \lt(\intd \Phi_p(\rho)\rt) &\leq (\lambda \C_{\ak,r})^{1/\eps} C_{\eps,\eps_0} M_0^{\theta_0(p+1)/\eps} \|\rho\|_{L^p}^{p(1+b_0)} + \lt(\eps_0 - \tfrac{\C_\al}{p}\rt) \|\rho\|_{L^{\tilde{r}}}^p,
		\end{align}
		where we take $\eps_0$ smaller than $\C_\al/p$. Since $1\leq p\leq \tilde{r}$, again by interpolation, we get
		\begin{equation*}
			\|\rho\|_{L^p}^{p(1+b_1)} \leq M_0^{p b_1} \|\rho\|_{L^{\tilde{r}}}^p,
		\end{equation*}
		with
		\begin{equation*}
			b_1 = \frac{\al}{d(p-1)}.
		\end{equation*}
		Thus, inequality~\eqref{eq:ineq_d_t_2} becomes
		\begin{align}\label{eq:ineq_d_t_3}
			\dt \|\rho\|_{L^p}^p &\leq C_1 M_0^{\theta_0(p+1)/\eps} \|\rho\|_{L^p}^{p(1+b_0)} - C_2 M_0^{-p b_1} \|\rho\|_{L^p}^{p(1+b_1)},
		\end{align}
		where $C_1 = (p-1)(\lambda \C_{\ak,r})^{1/\eps} C_{\eps,\eps_0}$ and $C_2 = (p-1)\lt(\tfrac{\C_\al}{p} - \eps_0\rt)$.
		\step{2. Conditions on $\eps$}
		We still have to verify that we can choose $\eps$ so that $(\theta_0,\theta_1)\in[0,1]^2$. By definition~\eqref{def:theta_1} of $\theta_1$, we get
		\begin{align*}
			\theta_1 \geq 0 &\,\ssi\, b_0 \geq -1
			\\
			&\,\ssi\, \ak-\al + \al \eps > -\eps d (p-1)
			\\
			&\,\ssi\, \eps \geq \frac{\al - \ak}{\al + d(p-1)} = \eps_m.
		\end{align*}
		Moreover, by definition~\eqref{def:theta_0} of $\theta_0$
		\begin{align*}
			\theta_0 \geq 0 &\,\ssi\, \theta_1 + \theta_2 \leq 1
			\\
			&\,\ssi\, \frac{p}{p+1}\lt(1+\eps b_0\rt) \leq 1
			\\
			&\,\ssi\, \eps b_0 \leq \frac{1}{p}
			\\
			&\,\ssi\, \frac{\ak-\al(1-\eps)}{d(p-1)} \leq \frac{1}{p}
			\\
			&\,\ssi\, \eps \leq 1-\frac{1}{\al}\lt(\ak-\frac{d}{q}\rt)  = \eps_M.
		\end{align*}
		Since $p < p_a$, $\eps_M < 1$. Let us check that it is nonnegative. We have
		\begin{align*}
			\eps_M \geq 0 &\,\ssi\, \ak-\frac{d}{q} \leq \al \,\ssi\, \frac{1}{q} \geq \frac{\ak-\al}{d}.
		\end{align*}
		Since $q=p'\geq1$, this is always verified when $\ak \leq \al$. When $\ak > \al$, it is verified by hypothesis since we can also read previous formula as
		\begin{equation*}
			\eps_M \geq 0 \,\ssi\, p \geq \frac{d}{d+\al-\ak} = p_{\ak,\al}.
		\end{equation*}
		When $\ak < \al$, we also have to verify that $\eps_m \leq \eps_M$. We have, indeed
		\begin{align*}
			\frac{\eps_M}{\eps_m} &= \frac{(p(\al-\ak)+d(p-1))(\al+d(p-1))}{p\al(\al-\ak)}
			\\
			&= \frac{p\al(\al-\ak) + d(p(p-1)(\al-\ak) + \al(p-1) + d(p-1)^2)}{p\al(\al-\ak)}
			\\
			&= 1 + d(p-1)\frac{p(\al-\ak)+\al+d(p-1)}{p\al(\al-\ak)} > 1.
		\end{align*}
		Therefore, since $\theta_2 \geq 0$ and $\theta_0 + \theta_1 + \theta_2 = 1$, we proved that for any $\eps\in[\max(\eps_m,0),\min(\eps_M,1)]$,
		\begin{equation*}
			(\theta_0,\theta_1,\theta_2)\in [0,1]^3.
		\end{equation*}
		By looking at \eqref{eq:ineq_d_t_3}, we want to take $\eps$ which minimizes $b_0$. Hence, we take
		\begin{align*}
			\eps = \eps_m &\text{ when } \ak < \al,
			\\
			\eps = \eps_M &\text{ when } \ak > \al.
		\end{align*}
		\step{3. Case $\ak < \al$}
		In this case, we have $\eps = \eps_m$, which yields $b_0 = -1$. Moreover, since
		\begin{equation*}
			\theta_0(p+1) = p+1 -(1-\eps) p = 1 + \eps p,
		\end{equation*}
		by \eqref{eq:ineq_d_t_3}, we obtain
		\begin{equation*}
			\dt \|\rho\|_{L^p}^p \leq C_1 M_0^{p+1/\eps} - C_2 M_0^{-p b_1} \|\rho\|_{L^p}^{p(1+b_1)}.
		\end{equation*}
		Then, either 
		\begin{equation}\label{eq:gronwall_1}
			C_2 M_0^{-p b_1} \|\rho\|_{L^p}^{p(1+b_1)} \leq 2 C_1 M_0^{p+1/\eps},
		\end{equation}
		or
		\begin{equation}\label{eq:gronwall_2}
			\dt \|\rho\|_{L^p}^p \leq- \tfrac{1}{2}C_2 M_0^{-p b_1} \|\rho\|_{L^p}^{p(1+b_1)}.
		\end{equation}
		Inequality~\eqref{eq:gronwall_1} is equivalent to
		\begin{equation*}
			\|\rho\|_{L^p}^p \leq \lt(\frac{2C_1}{C_2}\rt)^\frac{1}{1+b_1} M_0^{p+\frac{1}{\eps(1+b_1)}} =: C(M_0),
		\end{equation*}
		and by Gronwall's inequality, \eqref{eq:gronwall_2} leads to
		\begin{equation*}
			\|\rho\|_{L^p}^p \leq \lt(\tfrac{1}{2}C_2M_0^{-pb_1}b_1 t\rt)^{-1/b_1} = M_0^p \lt(\tfrac{b_1}{2}C_2t\rt)^{-1/b_1}
		\end{equation*}
		\step{4. Case $\ak > \al$}
		In this case, we have
		\begin{equation*}
			\eps = \eps_M = \frac{p\lt(\al-\ak\rt) + d\lt(p-1\rt)}{\al p} = \frac{1}{p\,b},
		\end{equation*}
		which by definition~\eqref{def:b_0} leads to
		\begin{align*}
			b_0 &= \frac{1}{d (p-1)}\lt(\al + \frac{\ak-\al}{\eps}\rt)
			\\
			&= \frac{\al}{d(p-1)}\lt(\frac{p(\al-\ak) + d(p-1) + p(\ak-\al)}{p(\al-\ak) + d(p-1)}\rt)
			\\
			&= \frac{\al}{p(\al-\ak) + d(p-1)} = b,
		\end{align*}
		and by inequality~\eqref{eq:ineq_d_t_3}, to
		\begin{equation*}
			\dt \|\rho\|_{L^p}^p \leq C_1 M_0^{\theta_0(p+1)/\eps} \|\rho\|_{L^p}^{p(1+b)} - C_2 M_0^{-p b_1} \|\rho\|_{L^p}^{p(1+b_1)}.
		\end{equation*}
		As remarked previously, $\eps = \eps_M \geq 0$. Therefore, since $\ak>\al$,
		\begin{equation}\label{eq:b-b_1}
			b = b_1 + \frac{\ak-\al}{\eps d(p-1)} \geq b_1.
		\end{equation}
		The estimate on the $L^p$ norm is then obtained by analysing the corresponding ODE which is of the form
		\begin{equation*}
			y'(t) = A\, y(t)^{1+b} - B\, y(t)^{1+b_1},
		\end{equation*}
		with $A$ and $B$ nonnegative. It has a fixed point at $y=0$ and at
		\begin{equation*}
			y^\sharp = \lt(\frac{B}{A}\rt)^\frac{1}{b - b_1} \geq 0.
		\end{equation*}
		Therefore, when $y(0) \in [0,y^\sharp)$, it yields $y(t) \in [0,y^\sharp)$ for any $t>0$, and since $y'\leq 0$ in this interval, it implies the existence of a constant $C = C(y(0)^\init) < 1$ such that
		\begin{equation*}
			A\,y^{1+b} \leq C B\, y^{1+b_1}.
		\end{equation*}
		It implies that
		\begin{equation*}
			y' \leq -(1-C)B\, y^{1+b_1},
		\end{equation*}
		which, by Gronwall's inequality, leads to
		\begin{equation*}
			\forall t\in\R_+, y \leq \frac{1}{(y(0)^{-b_1} + b_1 (1-C)Bt)^\frac{1}{b_1}} \leq \frac{M_0^p}{(b_1 (1-C)C_2t)^\frac{1}{b_1}}.
		\end{equation*}
		When $y(0) > y^\sharp$, we can still write that
		\begin{equation*}
			y' \leq A\, y^{1+b}.
		\end{equation*}
		It implies that the solution is bounded in $[0,T]$ for some $T>0$ and
		\begin{align*}
			\forall t\in [0,T], y(t) &\leq \frac{1}{(y(0)^{-b} - b At)^\frac{1}{b}}
			\\
			T &< \frac{1}{b Ay(0)^{b}}.
		\end{align*}
		We deduce the corresponding results for the $L^p$ norm of $\rho$ by Gronwall's inequality. When $y=y^\sharp$, all we get that $y$ is constant and therefore that $\|\rho\|_{L^p}^p \leq y^\sharp$ for any $t>0$. We can compute more precisely 
		\begin{align*}
			y^\sharp = \lt(\frac{C_2 M_0^{-pb_1}}{C_1 M_0^{\theta_0 (p+1)/\eps}}\rt)^\frac{1}{b-b_1} = \lt(\frac{C_2}{C_1}\rt)^\frac{1}{b-b_1} \lt(M_0^{-\theta_0 (p+1)/\eps-pb_1}\rt)^\frac{1}{b-b_1}.
		\end{align*}
		Now by the definitions of $C_1$ and $C_2$ in step $1$, by \eqref{eq:b-b_1} and the definition \eqref{def:theta_0} of $\theta_0$, we have
		\begin{align*}
			\theta_0(p+1) &= 1 - \eps p b = 0
			\\
			\lt(b-b_1\rt)\eps &= \frac{\ak-\al}{d(p-1)}
			\\
			C_1 &= (p-1) \lt(\lambda \C_{\ak,r}\rt)^{1/\eps} C_{\eps,\eps_0}
			\\
			C_2 &= (p-1)\lt(\tfrac{\C_\al}{p} - \eps_0\rt).
		\end{align*}
		This leads to
		\begin{align*}
			y^\sharp &= \lt(\frac{\C_\al-\eps_0p}{(\lambda \C_{\ak,r})^{1/\eps}C_{\eps,\eps_0}p}\rt)^\frac{1}{b-b_1} \lt(M_0^{p(b-b_1)-1/\eps}\rt)^\frac{1}{b-b_1}
			\\
			&= C_{\ak,\al,p}^p M_0^{p-\frac{d(p-1)}{\ak-\al}} \lambda^\frac{d(p-1)}{\ak-\al}.
		\end{align*}
		
		\step{5. Case $\ak=\al$}
		When $\ak = \al$, by definition~\eqref{def:b_0}, $b_0$ does not depend on $\eps$ and
		\begin{align*}
			b = b_0 = b_1 = \frac{\al}{d(p-1)}
			\\
			\theta_0 (p+1) = 1 - \eps p b.
		\end{align*}
		Moreover, we can take any $\eps\in (\eps_m,\eps_M] = (0,d/(\al q)]$. Thus, by inequality~\eqref{eq:ineq_d_t_3}, we get
		\begin{align*}
			\dt \|\rho\|_{L^p}^p &\leq C_1 M_0^{(1 - \eps p b)/\eps}\|\rho\|_{L^p}^{p(1+b)}  - C_2 M_0^{-p b}\|\rho\|_{L^p}^{p(1+b)}
			\\
			&\leq \|\rho\|_{L^p}^{p(1+b)} M_0^{-pb}\lt(C_1 M_0^{1/\eps}  - C_2\rt).
		\end{align*}
		The left hand side will be negative when
		\begin{equation}\label{eq:condition_masse}
			M_0 \leq \lt(\frac{C_2}{C_1}\rt)^{\eps} = \frac{\lt(\C_\al/p-\eps_0\rt)^\eps}{\lambda \C_{\ak,r}} \lt(\frac{\eps_0}{1-\eps}\rt)^{1-\eps}\eps^{-\eps} = u_\eps(\eps_0).
		\end{equation}
		Taking $\eps_0$ maximizing the right hand side, we get
		\begin{align*}
			\eps_0 &= (1-\eps)\C_\al / p
			\\
			u_\eps(\eps_0) &= \frac{\C_\al}{p\lambda \C_{\ak,r}} = \frac{C_{\ak,d,p}}{\lambda}.
		\end{align*}
		When this is the case, then $C^\init := |C_1 M_0^{1/\eps}  - C_2| > 0$ and by Gronwall's inequality
		\begin{equation*}
			\forall t\in\R_+, \|\rho\|_{L^p}^p \leq \frac{1}{(\|\rho^\init\|_{L^p}^{-pb} + b M_0^{-pb}C^\init t)^\frac{1}{b}} \leq \frac{M_0^p}{(bC^\init t)^\frac{1}{b}},
		\end{equation*}
		which proves \eqref{eq:etalement_critique}. When $M_0 > M_0^*$ we only get the existence of $T>0$ such that
		\begin{equation*}
			\forall t\in[0,T], \|\rho\|_{L^p}^p \leq \frac{1}{(\|\rho^\init\|_{L^p}^{-pb} - b M_0^{-pb}C^\init t)^\frac{1}{b}}.
		\end{equation*}
		Moreover, $T$ verifies
		\begin{equation*}
			T > \frac{1}{bC^\init} \lt(\frac{M_0}{\|\rho^\init\|_{L^p}}\rt)^{pb},
		\end{equation*}
		which proves \eqref{eq:estim_critique_temps_court}.
	\end{proof}
	
	\begin{cor}\label{cor:estim_Lp_tx}
		When $\ak < \al$ and $\rho^\init \in L^1$, then for any $p<p_\al$
		\begin{equation}\label{eq:estim_L1_tx_sscritique}
			\rho \in L^1_\mathrm{loc}(\R_+,L^p),
		\end{equation}
		which holds in particular if $p = p_\ak$. When $\ak > \al$ and $\rho^\init \in L^p$ for a given $p\in (p_{\ak,\al},p_a)$, then there exists $T>0$ such that
		\begin{equation}\label{eq:estim_Lp_tx_surcritique}
			\rho \in L^p((0,T),L^{\tilde{r}}),
		\end{equation}
		where $\tilde{r}= pp_\al \geq p_\ak$. Moreover, if \eqref{eq:etalement} is verified,
		\begin{equation*}
			\rho \in L^p_\mathrm{loc}(\R_+,L^{\tilde{r}}).
		\end{equation*}
	\end{cor}
	
	\begin{proof}
		Equation~\eqref{eq:estim_L1_tx_sscritique} comes from inequality~\eqref{eq:gain_Lp} by remarking that $p<p_\al$ implies $d/(\al q) \leq 1$ and integrating in time. Equation~\eqref{eq:estim_Lp_tx_surcritique} is a consequence of \eqref{eq:ineq_d_t_2}, which by integrating in time leads to
		\begin{equation*}
			\|\rho\|_{L^p}^p(t) + C_2 \int_0^t \|\rho(s)\|_{L^{\tilde{r}}}^p \dd s \leq \|\rho^\init\|_{L^p}^p + C_1 M_0^{\theta_0(p+1)/\eps} \int_0^t\|\rho(s)\|_{L^p}^{p(1+b_0)}\dd s.
		\end{equation*}
		If $\rho\in L^\infty([0,T],L^p)$, then we deduce that
		\begin{equation*}
			C_2 \int_0^t \|\rho(s)\|_{L^{\tilde{r}}}^p \dd s \leq \|\rho^\init\|_{L^p}^p + C_1 M_0^{\theta_0(p+1)/\eps} T \|\rho\|_{L^\infty([0,T],L^p)}^{p(1+b_0)},
		\end{equation*}
		and we conclude by using \eqref{eq:etalement} or \eqref{eq:estim_Lp_temps_court}.
	\end{proof}

\subsection{Tightness and coupling.}
	For the rest of the section we consider some given stochastic basis $\lt(\Omega, \mathcal{F},(\mathcal{F}_t)_{t\geq 0}, \mathbb{P}\rt)$. The expectation with respect to $\mathbb{P}$ will be denoted $\E$. We first provide a generalization of \cite[Proposition~3.1]{carrillo_propagation_2018} in the
	
	\begin{lem}\label{lem:loep}
		Let be $\ak\leq d$, $k\geq 1$ and $p \geq p_\ak$. There exists a constant $C$ depending only on $d,p,\ak$  such that for any $\rho_1,\rho_2\in \PP_k\cap L^p$ and $X,\bar{X}$ two i.i.d. random variables of law $\rho_1$ (respectively $Y,\bar{Y}$ two i.i.d. of law $\rho_2$), it holds when $p>p_\ak$
		\begin{equation}\label{eq:estim_d_wk_lineaire} \tag{i}
			\E\lt[ |X-Y|^{k-1} \lt| K(X-\bar{X}) - K(Y-\bar{Y}) \rt| \rt] \leq C C_{\rho_1,\rho_2} \Eps_k,
		\end{equation}
		and when $p = p_\ak$,
		\begin{equation}\label{eq:estim_d_wk_log}\tag{ii}
			\E\lt[ |X-Y|^{k-1} \lt| K(X-\bar{X}) - K(Y-\bar{Y}) \rt| \rt]
			\leq C C_{\rho_1,\rho_2} \Eps_k \lt( 1 + \frac{\ln_-(\Eps_k) }{k} \rt),
		\end{equation}
		where $C_{\rho_1,\rho_2} = 1+\|\rho_1\|_{L^p} + \|\rho_2\|_{L^p}$ and $\Eps_k = \E\lt|[X-Y|^k\rt]$.
	\end{lem}
	
	\begin{remark}
		The point $(i)$ of this Lemma has been extensively used in the literature (See for instance \cite{carrillo_local_2014, carrillo_derivation_2014, godinho_propagation_2015, salem_propagation_2017}). So has the point $(ii)$ in the Newtonian case $\ak=d$ and thus $p_\ak=\infty$ (see for instance \cite{loeper_uniqueness_2006, hauray_wasserstein_2009, carrillo_propagation_2018}). Since we did not found its generalization to a general Riesz interaction kernel $\ak\in(0,d)$, we provide more detail. A similar technique can be found in~\cite{lafleche_propagation_2019}.
	\end{remark}
	
	\begin{proof}
		We start with the classical inequality (see \cite[(3.9)]{hauray_wasserstein_2009}, \cite[Lemma 2.5]{godinho_propagation_2015}, \cite[(3.26)]{carrillo_derivation_2014}, \cite[(3.5)]{carrillo_local_2014}) which holds for any $(x,y)\in (\R^d)^2$
		\begin{equation*}
			\lt|K(x)-K(y)\rt| \leq \lt(\lt|\nabla K(x)\rt| + \lt|\nabla K(y)\rt|\rt) |x-y|.
		\end{equation*}
		Then denote $\pi=\L(X,Y)=\L(\bar{X},\bar{Y})\in \PP(\R^{2d})$.
		\step{1. Proof of \eqref{eq:estim_d_wk_lineaire}} We assume here that $p>p_\ak$. Then we have
		\begin{align*}
			\E&\lt[ |X-Y|^{k-1} \lt| K(X-\bar{X}) - K(Y-\bar{Y}) \rt|  \rt]
			\\
			&\quad\leq \E\lt[ |X-Y|^{k-1} (\lt|X-Y\rt|+\lt|\bar{X}-\bar{Y}\rt|) \lt( 	|\nabla K(X-\bar{X})|+|\nabla K(Y-\bar{Y})| \rt) \rt]
			\\
			&\quad :=\mathcal{I}_1+\mathcal{I}_2.
		\end{align*}
		We first estimate $\mathcal{I}_1$. Since $X$ and $\bar{X}$ are independent we get
		\begin{align*}
			\mathcal{I}_1 &= \E\lt[ |X-Y|^{k} \E\lt[ \lt( |\nabla K(X-\bar{X})|+|\nabla K(Y-\bar{Y})| \rt)| X,Y\rt] \rt]
			\\
			&= \E\lt[ |X-Y|^{k} \lt(\iintd  \lt( |\nabla K(X-x)|+|\nabla K(Y-y)| \rt)\pi(\d x,\d y) \rt) \rt]
			\\
			&= \E\lt[ |X-Y|^{k} \lt(\intd |\nabla K(X-x)|\rho_1(x)\dd x + \intd |\nabla K(Y-y)|\rho_2(y)\dd y \rt) \rt].
		\end{align*}
		But since $|\nabla K| \leq C_\ak |x|^{-\ak}$ with $C_\ak = \max(1-\ak,\ak)$, we obtain
		\begin{align*}
			C_{\ak}^{-1}\intd |\nabla K(X-x)|\rho_1(x)\dd x &\leq \intd|X-x|^{-\ak}\rho_1(x)\dd x
			\\
			&\leq \int_{|X-x|\leq r} |X-x|^{-\ak}\rho_1(x)\dd x + r^{-\ak}\|\rho_1\|_{L^1}
			\\
			&\leq \| \rho_1 \|_{L^p}\lt( \int_{|x|<r} |x|^{-\ak q}\dd x \rt)^{1/q} + r^{-\ak}\|\rho_1\|_{L^1},
		\end{align*}
		where $q=p'$ and $r>0$. Since $p>p_a$, we get $aq < d$ so that $|x|^{-aq}$ is locally integrable and we obtain
		\begin{align*}
			C_{\ak}^{-1}\intd |\nabla K(X-x)|\rho_1(x)\dd x &\leq C_K r^{d/q} \|\rho_1 \|_{L^p} + r^{-\ak}\|\rho_1\|_{L^1},
		\end{align*}
		where $C_K = \lt(\frac{\omega_d}{d-\ak q}\rt)^{1/q}$. 
						
		\step{2. Proof of \eqref{eq:estim_d_wk_log}} Note that for any $(x,y)\in (\R^d)^2$ and $r>0$, it holds
		\begin{equation*}
			\lt|K(x)-K(y)\rt| \leq
			\begin{cases}
			|K(x)|+|K(y)| & \text{if } |x|\wedge|y| \leq r
			\\ 
			\lt(\lt|\nabla K(x)\rt|+ \lt|\nabla K(y)\rt|\rt) |x-y| & \text{else}.
			\end{cases}
		\end{equation*}
		So that 
		\begin{align*}
			\E&\lt[ |X-Y|^{k-1} \lt| K(X-\bar{X}) - K(Y-\bar{Y}) \rt|  \rt] \leq
			\\
			&\quad \E\lt[ |X-Y|^{k-1} \lt( 	|K(X-\bar{X})|+|K(Y-\bar{Y})| \rt) \indic_{|X-\bar{X}|\wedge |Y-\bar{Y}|  \leq r} \rt]
			\\
			&\quad+ \E[ \lt|X-Y\rt|^{k-1} (\lt|X-Y\rt|+\lt|\bar{X}-\bar{Y}\rt|)
			\\
			&\qquad\qquad\qquad\lt( |\nabla K(X-\bar{X})|+|\nabla K(Y-\bar{Y})| \rt) \indic_{|X-\bar{X}|\wedge |Y-\bar{Y}| > r} ]
			\\
			&\quad =: \mathcal{I}_1+\mathcal{I}_2.
		\end{align*}
		To estimate $\mathcal{I}_1$, we write
		\begin{align*}
			\mathcal{I}_1 &= I^1_1+I^2_1+I^3_1
			\\
			&:= \E\lt[ |X-Y|^{k-1} \lt( 	|K(X-\bar{X})|+|K(Y-\bar{Y})| \rt)\indic_{|X-\bar{X}|\vee |Y-\bar{Y}|  \leq r}  \rt]
			\\
			&\quad + \E\lt[ |X-Y|^{k-1} \lt( 	|K(X-\bar{X})|+|K(Y-\bar{Y})| \rt)\indic_{|X-\overline{X}|> r \geq |Y-\overline{Y}|}  \rt]
			\\
			&\quad + \E\lt[ |X-Y|^{k-1} \lt( 	|K(X-\bar{X})|+|K(Y-\bar{Y})| \rt)\indic_{|X-\overline{X}| \leq r < |Y-\overline{Y}|}\rt].
		\end{align*}
		Then, for the estimate of $I_1^1$, we get by independence of $\bar{X}$ and $X$ (respectively $\bar{Y}$ and $Y$)
		\begin{align*}
			I_1^1&=\E\lt[ \E\lt[ \lt(	|K(X-\bar{X})|+|K(Y-\bar{Y})|\rt)\indic_{|X-\overline{X}|\vee |Y-\overline{Y}| \leq r} | X,Y \rt]  |X-Y|^{k-1}   \rt]
			\\
			&\leq \E\lt[\lt(\int_{|X-x|\leq r} \frac{\rho_1(x)}{|X-x|^{\ak-1}}\dd x+\int_{|Y-y|\leq r} \frac{\rho_2(y)}{|Y-y|^{\ak-1}}\dd y\rt) |X-Y|^{k-1}\rt]
			\\
			&\leq \lt( \|\rho_1\|_{L^{p_\ak} }+ \|\rho_2\|_{L^{p_\ak}}\rt) \lt( \int_{|z|\leq r} |z|^{-(\ak-1)\frac{d}{\ak}} dz\rt)^{\frac{\ak}{d}}   \E \bigl[ |X-Y|^{k-1}\bigr].
		\end{align*}
		Since
		\begin{equation*}
			\int_{|z|\leq r} |z|^{-(\ak-1)\frac{d}{\ak}} dz = \omega_d \int_{0}^r u^{d-1-(\ak-1)\frac{d}{\ak}}\dd s = \frac{\ak\omega_d}{d} r^\frac{d}{\ak} =: (C'_{d,\ak} r)^\frac{d}{\ak},
		\end{equation*}
		we get
		\begin{align*}
			I_1^1 &\leq C'_{d,\ak} r \lt(\|\rho_1\|_{L^{p_\ak}} + \|\rho_2\|_{L^{p_\ak}}\rt) \E\lt[|X-Y|^{k-1}\rt].
		\end{align*}
		For $I_1^2$, we have
		\begin{align*}
			I_1^2 &\leq \E\lt[\frac{2}{|Y-\overline{Y}|^{\ak-1}}\indic_{|Y-\overline{Y}|\leq r} |X-Y|^{k-1}\rt]
			\\
			&= 2\E\lt[\lt(\int_{|Y-y|\leq r} \frac{\rho_2(y)}{|Y-y|^{\ak-1}}\dd y\rt) |X-Y|^{k-1} \rt]
			\\
			&\leq 2C'_{d,\ak} r \lt\|\rho_2\rt\|_{L^{p_\ak}} \E\lt[|X-Y|^{k-1} \rt].
		\end{align*}
		We then estimate $I_1^3$ similarly. Combining the above estimates, we obtain
		\begin{align*}
			\mathcal{I}_1 &\leq 3C'_{d,\ak} r \lt( \|\rho_1\|_{L^{p_\ak}} + \|\rho_2\|_{L^{p_\ak}}\rt) \E\lt[ |X-Y|^{k-1} \rt]
			\\
			& \leq 3C'_{d,\ak} r \lt( \|\rho_1\|_{L^{p_\ak}} + \|\rho_2\|_{L^{p_\ak}}\rt)\E\lt[ |X-Y|^k \rt]^{(k-1)/k}.
		\end{align*}
		Next, we estimate $\mathcal{I}_2$ by writing
		\begin{align*}
			\mathcal{I}_2 &= C_a (I_2^1+I_2^2)
			\\
			&:= C_a \E\lt[|X-Y|^k \lt(\frac{1}{|X-\overline{X}|^\ak} + \frac{1}{|Y-\overline{Y}|^\ak}\rt) \indic_{|X-\overline{X}|\wedge|Y-\overline{Y}|> r}\rt]
			\\
			&\quad + C_a \E\lt[|X-Y|^{k-1}|\overline{X}-\overline{Y}|\lt(\frac{1}{|X-\overline{X}|^\ak}+\frac{1}{|Y-\overline{Y}|^\ak}\rt)\indic_{|X-\overline{X}|\wedge|Y-\overline{Y}|> r} \rt].
		\end{align*}
		First we easily obtain since $\indic_{a\wedge b \geq r}=\indic_{a\geq r}\indic_{b\geq r}$
		\begin{align*}
			I_2^1 &= \E\lt[ |X-Y|^k \E\lt[  \lt(\frac{1}{|X-\overline{X}|^\ak}+\frac{1}{|Y-\overline{Y}|^\ak}\rt) \indic_{|X-\overline{X}|\wedge|Y-\overline{Y}|> r} |X,Y \rt] \rt]
			\\
			&\leq \E\lt[|X-Y|^k \lt(\int_{|X-x|\geq r}\frac{\rho_1(x)}{|X-x|^\ak}\dd x + \int_{|Y-y|\geq r}\frac{\rho_2(y)}{|Y-y|^\ak}\dd y\rt)  \rt].
		\end{align*}
		We then consider two cases: $r > 1$ and $0 < r \leq 1$. For $r \leq 1$, we get
		\begin{align*}
			\int_{|X-x|\geq r}\frac{\rho_1(x)}{|X-x|^\ak}\dd x &= \int_{|X-x|> 1} \frac{\rho_1(x)}{|X-x|^\ak}\dd x + \int_{|X-x| \in [r,1]} \frac{\rho_1(x)}{|X-x|^\ak}\dd x
			\\
			&\leq \|\rho_1\|_{L^1} + \|\rho_1\|_{L^{p_\ak}} \lt(\int_{|X-x| \in [r,1]}\frac{1}{|X-x|^d}\dd x\rt)^{\frac{\ak}{d}}
			\\
			&\leq \|\rho_1\|_{L^1} + \omega_d\|\rho_1\|_{L^{p_\ak}}\ln_-(r)^{\frac{\ak}{d}}
			\\
			&\leq C_d \lt(\|\rho_1\|_{L^{p_\ak}}+\|\rho_1\|_{L^{1}}\rt)\lt(1 + \ln_{-} r \rt).
		\end{align*}
		For the case $r > 1$, it is clear to obtain
		\begin{equation*}
			\int_{|X-x|\geq r}\frac{\rho_1(x)}{|X-x|^\ak}\dd x \leq \|\rho_1\|_{L^1}.
		\end{equation*}
		This yields
		\begin{equation*}
			I_2^1 \leq C_d\lt(\|\rho_1\|_{L^{p_\ak}} + \|\rho_2\|_{L^{p_\ak}}+2\rt)\E\lt[|X-Y|^k\rt]\lt(1 + \ln_- r \rt).
		\end{equation*}
		On the other hand, by Hölder's inequality
		\begin{multline*}
			I_2^2 \leq \E\lt[|\overline{X}-\overline{Y}|^k\lt(\frac{1}{|X-\overline{X}|^\ak}+\frac{1}{|Y-\overline{Y}|^\ak}\rt)\indic_{|X-\overline{X}|\wedge|Y-\overline{Y}|> r}\rt]^{1/k}
			\\
			\times \E\lt[ |X-Y|^{k} \lt(\frac{1}{|X-\overline{X}|^\ak}+\frac{1}{|Y-\overline{Y}|^\ak}\rt) \indic_{|X-\overline{X}|\wedge|Y-\overline{Y}|> r} \rt]^{1-1/k}.
		\end{multline*}
		The second term of the product is some power of the term $I_2^1$ which has already been dealt with, and so is the second term by symmetry of the roles of $(X,Y)$ and $(\overline{X},\overline{Y})$. So that 
		\begin{equation*}
			\mathcal{I}_2 \leq C_{d,\ak} \lt(\|\rho_1\|_{L^{p_\ak}} + \|\rho_2\|_{L^{p_\ak}}+2\rt)\E\lt[|X-Y|^k\rt]\lt(1 + \ln_- r \rt).
		\end{equation*}
		Putting all these estimates together yields for any $r>0$
		\begin{align*}
			\E[ |X-Y|^{k-1} &\lt| K(X-\bar{X}) - K(Y-\bar{Y}) \rt|]
			\\
			&\leq C'_{d,\ak}  \lt( \|\rho_1\|_{L^{p_\ak}} + \|\rho_2\|_{L^{p_\ak}}\rt) r\E\lt[ |X-Y|^k \rt]^{1-1/k}
			\\
			&\quad + C_{d,\ak} \lt(\|\rho_1\|_{L^{p_\ak}} + \|\rho_2\|_{L^{p_\ak}}+2\rt)\E\lt[|X-Y|^k\rt]\lt(1 + \ln_- r \rt).
		\end{align*}
		Choosing $r= \E\lt[ |X-Y|^k \rt]^{1/k} $ yields the desired result.
	\end{proof}
		
	\begin{proof}({\bf Proof of Theorem~\ref{thm:welpos}}.)
		Let $\rho^\init$ be such as the assumptions of Theorem~\ref{thm:welpos}. For $\e>0$ define
		\begin{equation*}
			K_\eps(x) = \begin{cases}
			K(x) 			&\text{if } |x| \geq \eps\\ 
			\eps^{-\ak}x	&\text{else,}
			\end{cases}
		\end{equation*}
		and consider the following nonlinear PDE with smooth coefficient		
		\begin{equation}\label{FKSreg}
			\partial_t \rho_\eps = \I(\rho_\eps) + \lambda\divg((K_\eps*\rho_\eps)\rho_\eps),
		\end{equation}
		with the initial condition $\rho_{\eps}^\init=\rho^\init$. Since the kernel $K_\eps$ is ($\eps^{-\ak}$)-Lipschitz, the difficulty for the well posedness of \eqref{FKSreg} does not come from the quadratic nonlinear term. Existence and uniqueness of solution for this nonlinear problem is straightforward in the case $\ak \in (1,2)$. Indeed it is sufficient to apply a standard fixed point in $C([0,T],\PP_k)$ technique using Wasserstein metric, since in this case the solution a priori enjoys some $k\in (1,\ak)$ moment. In the case $\ak\in(0,1]$, it is no more possible to use the completeness of $C([0,T],\PP_\kappa), \kappa>1$, and we have to proceed by compactness (see \cite[Appendix B]{salem_propagation_2017}).
		 
	 	Then due to Proposition~\ref{LlnLest} (if $\al=\ak$), Corollary~\ref{cor:estim_Lp_tx} (if $\ak\neq \al$), and Proposition~\ref{prop:propag_m}, $\rho_\e\in L^1([0,T], L^{p})\cap L^{\infty}([0,T],L_k^1)$ for some $p\geq p_\ak$ and $T>0$ depending or not on $\rho^\init$, uniformly with respect to $\e>0$.
	 
		\step{1. Tightness} 		
		Let $X_0$ be a random variable on $\R^d$ of law $M_0^{-1}\rho^\init$ and $(Z_t^\al)_{t\geq 0}$ be an $\al$-stable Lévy process independent of $X_0$. We denote by $(X_t^\e)_{t\geq 0}$ (respectively $(X_t^{\e'})_{t\geq 0}$) the solution to the following SDE
		\begin{equation*}
			X_t^\e = X_0 -\lambda \int_0^t\intd K_\e(X_s^\e-x)\rho_\e(\d x)\dd s + Z_t^\al.
		\end{equation*}
		Note that $(\mu_\e(t))_{t\geq 0} := (\L(X_t^\e))_{t\geq 0}$ solves the linear PDE
		\begin{equation*}
			\partial_t\mu_\e = I(\mu_\e) +  \lambda\divg((K_\e*\rho_\e)\mu_\e),
		\end{equation*}
		with initial condition $\mu_\e^\init = M_0^{-1}\rho^\init$. Therefore $\L(X_t^\e)=M_0^{-1}\rho_\e(t)$ by uniqueness of solution to this linear PDE with smooth coefficient.
		
		Assume first $0<1-\ak<\al$. It is direct to obtain in this case for any $\gamma>1$
		\begin{align*}
			\KK_\e^\gamma :=& \iintd |K_\e(x-y)|^\gamma \rho_\e(\d x) \rho_\e(\d y) \\
			&\leq C_{\ak,\gamma}\iintd \lt(|x-y|\vee \e \rt) ^{\gamma(1-\ak)} \rho_\e(\d x) \rho_\e(\d y) \\
			&\leq C_{\ak,\gamma}\intd \lt( |x|^{(1-\ak)\gamma} + \e^{(1-\ak)\gamma} \rt)\rho_\e(\d x).
		\end{align*}
		Then choose $\gamma =\frac{k}{1-\ak}>1$ and use the symmetry between $x$ and $y$ to get
		\begin{align*}
			\sup_{0<\e<1}\int_0^T\KK_\e^\gamma(t) \dd t &\leq \sup_{0<\e<1}\int_0^T\iintd  C_{\ak,\gamma}\lt( |x|^{(1-\ak)\gamma} + \e^{(1-\ak)\gamma} \rt)\rho_\e(\d x) \rho_\e(\d y)\dd t
			\\ 
			&\leq C_{\ak,\gamma,T} \lt(\sup_{\e>0}\|\rho_\e\|_{L^\infty((0,T),L^1_k)}+1\rt) < \infty.
		\end{align*}
		Assume now that $\ak>1$. First note that Hardy-Littlewood-Sobolev inequality yields for any $\e>0$ and $\gamma>1$ to be fixed later
		\begin{align*}
			\KK_\e^\gamma \leq\iintd |x-y|^{-(\ak-1)\gamma} \rho_\e(\d x) \rho_\e(\d y) \leq C \|\rho_\e\|^2_{L^{\frac{d}{d+\gamma(1-\ak)/2}}}.
		\end{align*}
		By interpolation between Lebesgue spaces, if $\gamma <\frac{2(p-1)d}{\ak-1}$, then
		\begin{align*}
			\|\rho_\e\|_{L^{\frac{d}{d+\gamma(1-\ak)/2}}} &\leq \|\rho_\e\|^\theta_{L^{p}}\|\rho_\e\|^{1-\theta}_{L^1},
		\end{align*}
		where $\theta = \gamma \frac{(\ak-1)q}{2d}$ with $q=p'$. Therefore 
		\begin{align*}
			\sup_{\e>0} \int_0^T \KK_\e^\gamma(t) \dd t
			&\leq \sup_{\e>0}\int_0^T \|\rho_\e\|^2_{L^{\frac{d}{d+\gamma(1-\ak)/2}}}\dd t
			\\
			&\leq \sup_{\e>0}\int_0^T \|\rho_\e\|^{\gamma \frac{(\ak-1)q}{d}}_{L^{p}} \dd t < \infty,
		\end{align*}
		provided that $\gamma\in \lt(1,\frac{d}{(\ak-1)q}\rt)$. Then in both cases, denote the stochastic process
		\begin{equation*}
			J_t^\e = -\lambda\int_0^t \intd K_\e(X_s^\e-x)\rho_\e(\d x)\dd s,
		\end{equation*}
		and observe that for any $ 0\leq s<t\leq T$, it holds by Hölder's inequality
		\begin{align*}
			|J_t^\e-J_s^\e|&\leq \lt| \int_s^t  \intd K_\e(X_u^\e-x)\rho_\e(\d x)\dd u \rt|
			\\
			& \leq \int_s^t \intd   \lt|  K_\e(X_u^\e-x)\rt|    \rho_\e(\d x) \dd u
			\\
			& \leq |t-s|^{1/\gamma'} \int_0^T \lt( \intd |K_\e (X_u^\e-x) |^\gamma \rho_\e(\d x) \rt)^{1/\gamma}\dd u,
			\end{align*}
		so that by the estimates carried out in the beginning of this step  and Jensen's inequality
		\begin{align*}
			\sup_{0<\e<1}\E\lt[  \sup_{0\leq s<t\leq T}  \frac{|J_t^\e-J_s^\e|}{|t-s|^{1/\gamma'}}\rt]&
			\leq\int_0^T \E \lt[ \lt( \intd |K_\e (X_u^\e-x) |^\gamma \rho_\e(\d x) \rt)^{1/\gamma}\rt]\dd u
			\\
			&\leq \int_0^T \lt(\E \lt[  \intd |K_\e (X_u^\e-x) |^\gamma \rho_\e(\d x) \rt]\rt)^{1/\gamma}\dd u
			\\
			&\leq  T^{1/\gamma'} \lt(\int_0^T \KK_\e^\gamma(t) \dd t\rt)^{1/\gamma} < \infty.
		\end{align*}
		We then deduce that the family of law of the processes $(J_t^\e)_{t\in[0,T]}$ is tight in $\PP(C([0,T],\R^d))$. Indeed let us denote 
		\begin{equation*}
			\KK_{R} := \lt\{f\in C([0,T],\R^d), f(0)=0, \sup_{0\leq s<t\leq T} \frac{|f(t)-f(s)|}{|t-s|^{1/\gamma'}}\leq R \rt\},
		\end{equation*}
		which is compact due to Ascoli-Arzelà's Theorem. By Markov's inequality we get for any $\e>0$
		\begin{align*}
			\mathbb{P}\lt( (J_t^\e)_{0\leq t \leq T} \notin \KK_{R} \rt)
			&= \mathbb{P}\lt( \sup_{0\leq s<t\leq T}  \frac{|J_t^\e-J_s^\e|}{|t-s|^{1/\gamma'}} >R \rt)
			\\
			&\leq R^{-1}\sup_{1>\e>0}\E\lt[  \sup_{0\leq s<t\leq T}  \frac{|J_t^\e-J_s^\e|}{|t-s|^{1/\gamma'}}\rt].
		\end{align*}
		Hence the family of law of the processes $\L^\e=\L\lt((X_t^\e = X_0 + J_t^\e + Z_t^\al)_{0\leq t\leq T}\rt)\in \PP(D([0,T],\R^d))$ is tight. Thus, we can find a sequence $\e_n$ going to $0$ such that $\L_{\e_n}$ goes weakly to some $\pi\in \PP(D([0,T],\R^d))$. For any $t\in[0,T]$, we define $\mathbf{e}_t: g \in D([0,T],\R^d)\mapsto g(t)\in \R^d$ and $\rho(t):= (\mathbf{e}_t)\# \pi \in \PP$ the push-forward of $\rho$ by $\mathbf{e}_t$. Since for any $t\in[0,T]$, $(\mathbf{e}_t)\#\L^\e=\rho_\e(t)$, $\rho_{\e_n}(t)$ goes weakly to $\rho(t)$ in $\PP_k$,
		
		\step{2. A priori properties of the limit point}
		
		By lower semicontinuity of $\|\cdot\|_{L^{p}}$ and $\|\cdot\|_{L^1_k}$ with respect to the weak convergence of measures and Fatou's Lemma, it holds $\rho\in L^1([0,T], L^{p})\cap L^{\infty}([0,T],L_k^1)$. We now show that $\rho$ satisfies \eqref{eq:varFKS}. Indeed for $\varphi\in C^2_c$ denote
		\begin{align*}
			\F(\rho,t)=&\intd \lt(\rho(t)-\rho^\init\rt)\varphi - \int_0^t \intd \rho(s)  \lt(I(\varphi) - K_c*(\rho(s)\cdot\nabla\varphi)\rt)\dd s
			\\
			&- \int_0^t\iintd K_0(x-y)(\nabla\varphi(x)-\nabla\varphi(y))\rho(s,\d x)\rho(s,\d y)\dd s.
		\end{align*}
		Since $\rho_\e$ solves \eqref{FKSreg}, it holds for any $t>0$
		\begin{equation*}
			\F_\e(\rho_\e,t)=0,
		\end{equation*}
		where $\F_\e$ is the same functional as $\F$ with $K$ replaced with $K_\e$. So that for any $t\in[0,T]$
		\begin{equation*}
			\lt|\F(\rho,t)\rt| \leq \lt|\F(\rho,t)-\F_\eta(\rho,t)\rt| + \lt|\F_\eta(\rho,t)-\F_\eta(\rho_\e,t)\rt| + \lt|\F_\eta(\rho_\e,t)-\F_\e(\rho_\e,t)\rt|.
		\end{equation*}
		But note that for $\eta > \e \geq 0$
		\begin{equation*}
			\lt|K_\e(x)-K_{\eta}(x)\rt| \leq \indic_{\e\leq |x|\leq \eta}|x|^{1-\ak} \leq \eta |x|^{-\ak}.
		\end{equation*}
		We deduce that for any $\varrho\in L^1([0,T]; L^{p_\ak})$, by \eqref{eq:HLS}, it holds 
		\begin{align*}
			\lt|\F_\eta(\varrho,t)-\F_\e(\varrho,t) \rt|&\leq  \eta  \int_0^t\iint  |x-y|^{-\ak}\varrho_s(\d x)\varrho_s(\d y)\dd s
			\\
			&\quad \leq \eta\, \CHLS_{d,\ak,p_{\ak/2}} \int_0^t \|\varrho\|^2_{L^{\frac{2d}{2d-\ak} } }\d s\leq \eta\, \CHLS_{d,\ak,p_{\ak/2}} \int_0^t\|\varrho\|_{L^{p_\ak}}\dd s.
		\end{align*}
		So that
		\begin{align*}
			\lt|\F(\rho,t)\rt|\leq \,&\eta\, \CHLS_{d,\ak,p_{\ak/2}} \lt( \int_0^t\|\rho\|_{L^{p_\ak}}\dd s+\sup_{0<\e<1} \int_0^t\|\rho_\e\|_{L^{p_\ak}}\dd s\rt) 
			\\
			&+ |\F_\eta(\rho,t)-\F_\eta(\rho_\e,t)|.
		\end{align*}
		Letting first $\e$ go to $0$ makes the second term in the r.h.s. vanish, since for fixed $\eta>0$, $\F_\eta$ is a smooth function on $L^1([0,T]; L^{p_\ak})$ and $\rho_\e$ goes weakly to $\rho$ as $\e$ goes to $0$, then letting $\eta$ go to $0$ yields $\F(\rho,t) = 0$, and $\rho$ is a solution to the \eqref{eq:FKS} equation in the sense of Definition \ref{def:sol}.
		
		\step{3. Uniqueness of the limiting point}
		
		We now show that there exists at most one such solution. Let $\rho$, $\tilde{\rho}\in L^1([0,T], L^{p})\cap L^{\infty}([0,T],L_k^1)$ for some $p\geq p_\ak$ and $T>0$ be two solutions to the \eqref{eq:FKS} equation with initial condition $\rho^\init$. We argue by a coupling argument. Define 
		\begin{align*}
			X_t := X_0 -\lambda \int_0^t \intd K(X_s-y)\rho(\d y)\dd s + Z_t^\al
			\\
			Y_t := X_0 -\lambda \int_0^t \intd K(Y_s-y)\tilde{\rho}(\d y)\dd s + Z_t^\al.
		\end{align*}
		Due to the $L^{p}$ regularity of $\rho$ and $\tilde{\rho}$ and Lemma~\ref{lem:loep}, $K*\rho$ and $K*\tilde{\rho}$ are Lipschitz if $p>p_\ak$ and log-Lipschitz if $p=p_\ak$. But $\mu(t) := \L(X_t)$ solves the linear PDE
		\begin{equation*}
			\partial_t \mu = \I(\mu) + \lambda\divg\lt((K*\rho)\mu\rt),
		\end{equation*}
		for the initial condition $\mu(0) = M_0^{-1}\rho^\init$. By uniqueness of solution to this linear PDE with Lipschitz or log-Lipschitz coefficient, $\L(X_t)=M_0^{-1}\rho(t)$ (respectively $\L(Y_t)=M_0^{-1}\tilde{\rho}(t)$). Denoting $Z_s=X_s-Y_s$, and $\pi_s=\L(X_s,Y_s)$ yields
		\begin{align*}
			|Z_t|^2 &= -2\lambda \int_0^t \iintd Z_s\cdot (K(X_s-x)-K(Y_s-y)) \pi_s(\d x,\d y) \dd s.
		\end{align*}
		Introducing $\bar{X}_s$ i.i.d. from $X_s$ (respectively $\bar{Y}_s$ i.i.d. from $Y_s$) and taking the expectation yields
		\begin{align*}
			\E\lt[|Z_t|^2\rt]&\leq 2\lambda\int_0^t \E\lt[|Z_s| |K(X_s-\bar{X}_s)-K(Y_s-\bar{Y}_s) | \rt] \d s
			\\
			&\leq \lt\{\begin{array}{ll}
			\!\!\!C \int_0^t \lt(\|\rho \|_{L^p}+\|\tilde{\rho}\|_{L^p}+2\rt) \E\lt[ |Z_s|^2  \rt]\d s, &\text{if } p> p_\ak
			\\ 
			\!\!\!C \int_0^t\lt(\|\rho\|_{L^{p_\ak}}+\|\tilde{\rho}\|_{L^{p_\ak}}+2\rt) \E\lt[ |Z_s|^2  \rt]\lt( 1+\frac{\ln^{-} \lt(\E\lt[|Z_s|^2\rt] \rt) }{2} \rt)\d s &\text{else.}
			\end{array}\rt.
		\end{align*}
		where we used Lemma~\ref{lem:loep}. By Gronwall's inequality, we get
		\begin{equation*}
			\forall t\in[0,T], \E\lt[|Z_t|^2\rt]=0, \text{ i.e. } \forall t\in[0,T], \rho(t)=\tilde{\rho}(t),
		\end{equation*}
		which yields the desired results.
	\end{proof}

\section{Proof of Theorem~\ref{thm:BU}}\label{sec:blowup}

	We first study the local and asymptotic space behaviour of the fractional Laplacian of some basic functions.
	
	\begin{lem}
		Let $\varphi\in C^\infty_c$ be such that $\intd \varphi = 1$. Then for any $\ak>\al$
		\begin{equation}\label{eq:lap_pol_surcritique}
			\lt|\I(|x|^\ak\varphi)\rt| \leq C\weight{x}^{-(d+\al)}.
		\end{equation}
	\end{lem}
	
	\begin{proof}
		Let $\varphi_\ak := |x|^\ak \varphi$ and $R>0$ be such that $\mathrm{supp}(\varphi)\subset\ball_R$. Then, for any $x\in\ball_R^c$, we obtain
		\begin{equation}\label{eq:estim_I_phi_ak}
			\I(\varphi_\ak)(x) = \int_{\ball_R} \frac{\varphi_a(y)\dd y}{|x-y|^{d+\al}} \in \lt(\frac{m_{\varphi_\ak}}{(|x|+R)^{d+\al}},\frac{m_{\varphi_\ak}}{(|x|-R)^{d+\al}}\rt).
		\end{equation}
		Now, assume $x\in\ball_r$ for a given $r>R$. Then we write the fractional Laplacian as 
		\begin{equation*}
			\I(\varphi_\ak) = \intd h_{\al,\ak}(y)\dd y,
		\end{equation*}
		where
		\begin{align*}
			h_{\al,\ak}(y) &= \frac{\varphi_\ak(y)-\varphi_\ak(x)}{|x-y|^{d+\al}} &&\text{when } \al\in(0,1)
			\\
			h_{\al,\ak}(y) &= \frac{\varphi_\ak(y)-\varphi_\ak(x)-(y-x)\cdot\nabla\varphi_\ak(x)}{|x-y|^{d+\al}} &&\text{when } \al\in[1,2).
		\end{align*}
		Then since $\varphi_\ak\in W^{\ak,\infty}$, we obtain that $h_{\al,\ak}(y)\leq C |x-y|^{-d+\ak-\al}$, which, since $\ak>\al$, implies that $h_{\al,\ak}\in L^1_\mathrm{loc}$. Moreover, when $|y|>r$, then 
		\begin{equation*}
			h_{\al,\ak}(y) \leq \frac{C_{\varphi}r^\al}{(|y|-r)^{d+\al}} \in L^1(B_r^c).
		\end{equation*}
		Therefore, $h_{\al,\ak}\in L^1$ uniformly in $x\in B_r$. Hence $\I(\varphi_\ak)\in L^\infty(B_r)$, which, combined with \eqref{eq:estim_I_phi_ak}, leads to the expected result.
	\end{proof}
	
	\begin{lem}
		Let $\varphi\in C^\infty_c$ be such that $\intd \varphi = 1$ and $\mathds{1}_{\ball_r} \leq \varphi \leq \mathds{1}_{\ball_{2r}}$. Then for any $k\in(0,\al)$
		\begin{equation}\label{eq:lap_pol_souscritique}
			\lt|\I(|x|^k\varphi^c)\rt| \leq C\weight{x}^{k-\al},
		\end{equation}
		where $\varphi^c = 1-\varphi$.
	\end{lem}
	
	\begin{proof}
		The proof is a straightforward adaptation of \cite[Remark~4.2]{biler_blowup_2010} for $k>1$ and \cite[Proposition~2.2]{lafleche_fractional_2018} for $k<1$.
	\end{proof}
	
	We are now ready to prove the finite time blow-up.
	
	\begin{proof}({\bf Proof of Theorem~\ref{thm:BU}}.)
		Let $\varphi\in C^\infty_c(\R)$ even and nonincreasing be such that $\int_\R \varphi = 1$ and $\mathds{1}_{\ball_r} \leq \varphi\leq \mathds{1}_{\ball_{2r}}$ for a given $r\in(0,1/2)$ and $\varphi^c = 1-\varphi$. We define
		\begin{equation*}
			m(x) := \varphi(|x|)|x|^\ak + \varphi^c(|x|)|x|^k.
		\end{equation*}
		Assuming the existence of $\rho\in L^\infty((0,T),L^1_k)$ to the \eqref{eq:FKS} equation, we get
		\begin{equation}
		\label{eq:diffmom}
		\begin{aligned}
			\dt \lt(\intd \rho m\rt) & = \intd \rho \I(m) - \lambda\iintd \frac{(\nabla m(x)-\nabla m(y))\cdot(x-y)}{|x-y|^\ak}\rho(\d x)\rho(\d y)
			\\
			&=\mathcal{I}_1-\lambda\,\mathcal{I}_2,
		\end{aligned}
		\end{equation}	

		\noindent\textit{\textbf{$\bullet$ Estimate of $\mathcal{I}_1$.}} By the inequalities \eqref{eq:lap_pol_surcritique} and \eqref{eq:lap_pol_souscritique}, we get
		\begin{equation}\label{eq:estim_I_m_mix}
		\I(m) \leq C\weight{x}^{k-\al}.
		\end{equation}
		Hence, for some constant $C_1>0$, the following inequality holds
		\begin{equation*}
			\mathcal{I}_1\leq C \intd \weight{x}^{k-\al} \rho  \leq C_1M_0.
		\end{equation*}
		
		\noindent\textit{\textbf{$\bullet$ Estimate of $\mathcal{I}_2$.}}
		
		\noindent$\diamond$ \textit{Step one: case $1< k < \al < \ak$}. In this case, by convexity we have for any $(x,y)\in \R^d\times\R^d$
	\begin{equation*}
		(\nabla m(x)-\nabla m(y))\cdot(x-y) = g(x,y)-h(x,y)\,x\cdot y \geq 0
	\end{equation*}
	with $m'(|x|) = \nabla m(x)\cdot \frac{x}{|x|}$ and
		\begin{align*}
			g(x,y) &= m'(|x|)|x|+m'(|y|)|y|
			\\
			h(x,y) &= m'(|x|)|x|^{-1}+m'(|y|)|y|^{-1}.
		\end{align*}
		Since $|x-y|^\ak \leq 2^\ak(|x|^\ak+|y|^\ak)$, we obtain
		\begin{align*}
		\mathcal{I}_2=\iintd \frac{g(x,y) - h(x,y) \lt(x\cdot y\rt)}{|x-y|^\ak}\rho(\d x)\rho(\d y) \geq  &\iintd \frac{g(x,y)}{2^{\ak}(|x|^\ak+|y|^\ak)}\rho(\d x)\rho(\d y)\\
		&-\iintd \frac{ h(x,y) x\cdot y}{2^{\ak}(|x|^\ak+|y|^\ak)}\rho(\d x)\rho(\d y).
		\end{align*}
		Next since
		\[
		\iintd \frac{ h(x,y) x\cdot y}{2^{\ak}(|x|^\ak+|y|^\ak)}\rho(\d x)\rho(\d y)=\intd \lt(\intd\frac{ h(x,y) x\cdot y}{2^{\ak}(|x|^\ak+|y|^\ak)}\rho(\d x)\rt)\rho(\d y),
		\]
		by Fubini's theorem, and since for any $y\in \R^d$ the map $x\mapsto \frac{h(x,y) x\cdot y}{2^{\ak}(|x|^\ak+|y|^\ak)}$ is odd and $\rho$ is even, we get 
		\begin{align}\label{eq:int_g}
			\mathcal{I}_2 &\geq  \iintd \frac{g(x,y)}{2^{\ak}(|x|^\ak+|y|^\ak)}\rho(\d x)\rho(\d y).
		\end{align}
		We remark that if $(x,y)\in \ball_r^2$,
		\begin{equation*}
			\frac{g(x,y)}{2^{\ak}(|x|^\ak+|y|^\ak)} = \frac{\ak}{2^{\ak}}.
		\end{equation*}
		If $(x,y)\in (\ball_{2r}^c)^2$,
		\begin{equation*}
			\frac{g(x,y)}{2^{\ak}(|x|^\ak+|y|^\ak)} = \frac{k(|x|^k+|y|^k)}{2^{\ak}(|x|^\ak+|y|^\ak)} \geq \frac{k(2r)^{\ak-k}}{2^{\ak}(|x||y|)^{\ak-k}}.
		\end{equation*}
		If $(x,y)\in \ball_{r}\times \ball_{2r}^c$,
		\begin{equation*}
			\frac{g(x,y)}{2^{\ak}(|x|^\ak+|y|^\ak)} = \frac{\ak|x|^\ak+k|y|^k}{2^{\ak}(|x|^\ak+|y|^\ak)} \geq \frac{k|y|^k}{2^{\ak}(r+|y|^\ak)}.
		\end{equation*}
		Moreover, when $x\in \ball_{2r}\backslash\ball_r$,
		\begin{equation*}
			m'(|x|)|x| = \varphi'(|x|)(|x|^{\ak+1} - |x|^{k+1}) + \ak\varphi(|x|)|x|^{\ak} + k\varphi^c(|x|) |x|^k.
		\end{equation*}
		Remarking that we can take $\varphi$ decreasing and $r<1/2$, which implies that $|x|\leq 1$ and
		\begin{equation*}
			m'(|x|)|x| \geq \ak\varphi(|x|)|x|^{\ak} + k\varphi^c(|x|) |x|^k \geq k|x|^\ak,
		\end{equation*}
		it allows us to do the same kind of estimates for the remaining $(x,y)\in\R^{2d}$ and obtain
		\begin{equation}\label{eq:estim_g_below}
			\frac{g(x,y)}{2^{\ak}(|x|^\ak+|y|^\ak)} \geq C\weight{x}^{k-\ak}\weight{y}^{k-\ak}.
		\end{equation}
		Combining \eqref{eq:estim_I_m_mix}, \eqref{eq:int_g} and \eqref{eq:estim_g_below}, we obtain
		\begin{align}\nonumber
			\dt \lt(\intd \rho\,m\rt) &\leq C_1\intd \rho \weight{x}^{k-\al} - C_2 \lambda\iintd \weight{x}^{k-\ak}\weight{y}^{k-\ak}\rho(\d x)\rho(\d y)
			\\\label{eq:ineq_blowup}
			&\leq C_1 M_{k-\al} - C_2 \lambda M_{k-\ak}^2,
		\end{align}
		where $M_k = \intd \rho \weight{x}^k$. We define
		\begin{equation*}
			Y := M_0 + \intd \rho m  = \intd \rho \lt(1 + m\rt).
		\end{equation*}
		Remarking that 
		\begin{equation*}
			\frac{1}{2}(1+m)\leq \weight{x}^k \leq 2^{k/2}(1+m),
		\end{equation*}
		we obtain that $Y$ can always be compared to $M_k$ up to a constant depending on $k$. Therefore, Hölder's inequality yield
		\begin{align*}
			M_0 \leq M_{k-\ak}^\frac{k}{\ak} M_k^{1-\frac{k}{\ak}} \leq C M_{k-\ak}^\frac{k}{\ak} Y^{1-\frac{k}{\ak}}.
		\end{align*}
		Thus, using the fact that $M_{k-\al} < M_0$ because $k-\al < 0$ and the conservation of the total mass $M_0$, we obtain
		\begin{equation*}
			\frac{\d Y}{\d t} \leq C_1 M_0 - C'_2 \lambda M_0^\frac{2\ak}{k} Y^{2\lt(1-\frac{\ak}{k}\rt)}.
		\end{equation*}
		By assumption~\eqref{eq:condBU_1} for the appropriate $C^*$,
		\begin{equation*}
			\eps := C_2\lt(1 - \frac{C_1Y^{2(\frac{\ak}{k}-1)}(0)}{C_2\lambda M_0^{\frac{2\ak}{k}-1}}\rt) > 0.
		\end{equation*}
		Then for any $t\geq0$, $\frac{\d Y}{\d t}\leq 0$ and
		\begin{equation*}
			Y^{2(\frac{\ak}{k}-1)}(t) \leq Y^{2(\frac{\ak}{k}-1)}(0) = \frac{C_2-\eps}{C_1}\lambda M_0^{\frac{2\ak}{k}-1},
		\end{equation*}
		and
		\begin{align*}
			\frac{\d Y}{\d t} &\leq -\eps \lambda M_0^\frac{2\ak}{k} Y^{2\lt(1-\frac{\ak}{k}\rt)}.
		\end{align*}
		By Gronwall's inequality, we deduce
		\begin{equation*}
			Y(t) \leq \lt(Y(0)^{\frac{2\ak}{k}-1} - \eps \lambda \lt(\tfrac{2\ak}{k}-1\rt) M_0^\frac{2\ak}{k} t\rt)^\frac{k}{2\ak-k}.
		\end{equation*}
		Since $Y$ is positive and the above inequality goes to $0$ in finite time, we deduce that the solution ceases to be well defined in $L^1$ in a finite time $T^*$ verifying
		\begin{equation*}
			T^* < \frac{kY(0)^{\frac{2\ak}{k}-1}}{\eps \lambda (2\ak-k) M_0^\frac{2\ak}{k}}
			= \frac{k}{2\ak-k}\frac{Y(0)^{\frac{2\ak}{k}-1}}{C_2\lambda M_0^\frac{2\ak}{k} - C_1 Y^{2(\frac{\ak}{k}-1)}(0)M_0},
		\end{equation*}
		which proves the result.
		
		\noindent$\diamond$ \textit{Step two: Case $0 < k < \al \leq 1 \leq \ak$.} We use the symmetry between $x$ and $y$ to rewrite 
		\begin{align*}
		\mathcal{I}_2 &\geq C \lt( \iint_{ \substack{|x|\leq r\\ |y|\leq r} } + \iint_{  \substack{|x|> r\\ |x-y|\leq r/2} } + \iint_{  \substack{|x|> r\\ |x-y|> r/2 \\ |y|<r} } + \iint_{  \substack{|x|> r\\ |x-y|> r/2 \\ |y|>r} }  \rt) 
		\\
		&\quad \quad \quad \frac{\lt(\nabla m(x)-\nabla m(y)\rt)\cdot (x-y)}{|x-y|^\ak}\rho(\d x)\rho(\d y)\\
		&=\mathcal{I}_2^1 + \mathcal{I}_2^2 + \,\mathcal{I}_2^3 + \,\mathcal{I}_2^4.
		\end{align*}
		
		$\bullet$ Estimate of $\mathcal{I}_2^1$. For $|x|\leq r$, we have $m(x)=|x|^\ak$. Hence by strict convexity (since $\ak\geq 1$), we expand the inner product similarly as in the beginning of step one to obtain, with the same arguments 
		\begin{align*}
			\mathcal{I}_2^1=&\iint_{\substack{|x|\leq r\\|y|\leq r}} \frac{g(x,y) - h(x,y) \lt(x\cdot y\rt)}{|x-y|^\ak}\rho(\d x)\rho(\d y)\\
			 &\geq \iint_{\substack{|x|\leq r\\|y|\leq r}} \frac{g(x,y)-h(x,y)x\cdot y}{2^{\ak}(|x|^\ak+|y|^\ak)}\rho(\d x)\rho(\d y)
			\\
			&= \iint_{\substack{|x|\leq r\\|y|\leq r}} \frac{g(x,y)}{2^{\ak}(|x|^\ak+|y|^\ak)}\rho(\d x)\rho(\d y)
			\\\nonumber
			&\geq \frac{\ak}{2^\ak}\lt(\int_{B_r} \rho\rt)^2.
		\end{align*}
		
		$\bullet$ Estimate of $\mathcal{I}_2^2$. We may choose the linking function $\varphi$ in the definition of $m$ smooth enough so that  for $|x|>r$ it holds $|\nabla m(x)| \leq C |x|^{k-1}$. And since $k-1\leq 0$ and  $1-\ak \leq 0$, we have
		\begin{align*}
			\mathcal{I}_2^2 &=\iint_{\substack{|x|>r\\|y|> r\\ |x-y|> r/2}} \frac{\lt(\nabla m(x)-\nabla m(y)\rt)\cdot (x-y)}{|x-y|^\ak}\rho(\d x)\rho(\d y)
			\\
			&\geq - C \iint_{\substack{|x|>r\\|y|> r\\ |x-y|> r/2}} \lt(|x|^{k-1}+|y|^{k-1}\rt) |x-y|^{1-\ak} \rho(\d x)\rho(\d y)
			\\
			&\geq - 2^\ak C r^{k-\ak} \lt(\int_{B_r^c} \rho\rt)^2\!,
		\end{align*}
		
		$\bullet$ Estimate of $\mathcal{I}_2^3$. Similar considerations yield
		\begin{align*}
			\mathcal{I}_2^3 &=\iint_{\substack{|x|>r\\|y|\leq r\\ |x-y|> r/2}} \frac{\lt(\nabla m(x)-\nabla m(y)\rt)\cdot (x-y)}{|x-y|^\ak}\rho(\d x)\rho(\d y)
			\\
			&\geq - 2^{\ak-1} C \lt(r^{k-\ak} + 1\rt) \lt(\int_{B_r^c} \rho\rt)\lt(\int_{B_r} \rho\rt).
		\end{align*}

		$\bullet$ Estimate of $\mathcal{I}_2^4$. When $|x-y|\leq r/2$ and $|x|>r$, remark that it holds
		\begin{align*}
			|x| \,\leq\, |x-y| + |y| &\,\leq\, \frac{r}{2} + |y| \,<\, \frac{|x|}{2} + |y|
			\\
			|y| \,\leq\, |x-y| + |x| &\,\leq\, \frac{r}{2} + |x| \,\leq\, \frac{3|x|}{2}
		\end{align*}
		which implies that $r\leq \lt|x\rt| \leq 2\lt|y\rt| \leq 3\lt|x\rt|$. Therefore, we can write
		\begin{align*}
			\mathcal{I}_2^4 &=\iint_{\substack{|x|>r\\ |x-y|\leq r/2}} \frac{\lt(\nabla m(x)-\nabla m(y)\rt)\cdot (x-y)}{|x-y|^\ak}\rho(\d x)\rho(\d y)
			\\
			&\geq -C \iint_{\substack{|x|>r\\ |x-y|\leq r/2\\ |y|> r/2}} \lt||x|^{k-2}x-|y|^{k-2}y\rt| \lt|x-y\rt|^{1-\ak} \rho(\d x)\rho(\d y).
		\end{align*}
		Then, since $\lt|\nabla |z|^{k-2}\rt| = \lt|k-2\rt| |z|^{k-3}$ and $2\lt|x-y\rt|\leq |x|$, we obtain
		\begin{align*}
			\lt||x|^{k-2}x-|y|^{k-2}y\rt| &\leq \lt|\lt(|x|^{k-2}-|y|^{k-2}\rt)x\rt| + \lt||y|^{k-2}(x-y)\rt|
			\\
			&\leq C_k \lt|x\rt| \lt|x-y\rt| \sup_{|z|\geq |x|/2} |z|^{k-3} + |y|^{k-2} |x-y|
			\\
			&\leq C_{\ak,k}\, r^{k-\ak} \lt|x-y\rt|^{\ak-1},
		\end{align*}
		from which we get
		\begin{align*}
			\mathcal{I}_2^2 &\geq -C_{\ak,k} r^{k-\ak} \lt(\int_{B_{r/2}^c} \rho\rt)^2.
		\end{align*}
		Defining $Y = \intd \rho m$ and using the fact that
		\begin{align*}
			\tilde{Y} := \int_{B_r^c} \rho \,\leq\, C_{m,r}\, Y \text{  and  } \int_{B_r} \rho \,=\, M_0 - \tilde{Y},
		\end{align*}
		and gathering the previous estimates yields the existence of positive constants $C_2$, $C_3$, depending on $\ak$, $k$ and $r$ such that
		\begin{align*}
		\mathcal{I}_2&\geq \frac{\ak}{2^\ak}\lt( M_0 - \tilde{Y}\rt)^2-C_{\ak,k} r^{k-\ak} \tilde{Y}^2 -C 2^\ak r^{k-\ak} \tilde{Y}^2-C 2^{\ak-1} \lt(r^{k-\ak} + 1\rt) \tilde{Y}\lt( M_0 - \tilde{Y}\rt)
		\\
		&\geq \frac{C_2}{2}\lt( M_0^2-C_3 \tilde{Y}^2 \rt).
		\end{align*}
		Coming back to~\eqref{eq:diffmom} and using the fact that $\tilde{Y} \leq C_{m,r} Y$ yields the existence of a constant $C_4$ such that
		\begin{align*}
			\frac{\d Y}{\d t} &\leq C_1 M_0 + \frac{C_2}{2} \lambda \lt(C_4 Y^2 - M_0^2\rt).
		\end{align*}
		In particular, as long as $Y^2 \leq (2\,C_4)^{-1}M_0^2$ and $C_2\lambda M_0 \geq 8\, C_1$ it holds
		\begin{align}\label{eq:ineq_dY_dt}
			\frac{\d Y}{\d t} \leq C_1 M_0 - \frac{C_2}{4} \lambda M_0^2 \leq - C_1 M_0 \leq 0.
		\end{align}
		In particular, if $Y(0)^2 \leq (2\,C_4)^{-1}M_0^2$ then $Y$ remains decreasing for all times and for all $t>0$, $Y(0)^2 \leq (2C_3)^{-1}M_0^2$. By using again \eqref{eq:ineq_dY_dt}, this implies
		\begin{align*}
			Y(t) &\leq Y(0) - C_1M_0\, t,
		\end{align*}
		which becomes negative in finite time and leads again to a contradiction. The fact that the condition~\eqref{eq:condBU_2} is sufficient comes from the fact that there exists a constant $C>0$ such that 
		\begin{equation*}
			Y = \intd \rho m \leq C \intd \rho(x) |x|^k\dd x,
		\end{equation*}
		since $k<\ak$.
	\end{proof}

\section*{Acknowledgements}\label{sec:acknowledgements}

The second author was supported by the Fondation des Sciences Mathématiques de Paris and Paris Sciences \& Lettres Université.


\renewcommand{\bibname}{\centerline{Bibliography}}
\bibliographystyle{abbrv} 
\bibliography{KellerSegel}

\end{document}